\documentclass[11pt]{article}
\usepackage{amsmath,amsfonts,amssymb, amsthm,enumerate,graphicx}
\usepackage{tikz,authblk}
\usepackage{subfig}
\usepackage{url}
\usepackage[nameinlink]{cleveref}

\newcommand{\lk}[2]{{\rm lk}_{#1}(#2)}
\newcommand{\st}[2]{{\rm st}_{#1}(#2)}

\newtheorem{Lemma}{Lemma}[section]
\newtheorem{Theorem}[Lemma]{Theorem}
\newtheorem{Proposition}[Lemma]{Proposition}
\newtheorem{Corollary}[Lemma]{Corollary}
\newtheorem{Remark}[Lemma]{Remark}
\newtheorem{definition}[Lemma]{Definition}

\def\vst{\vskip .1cm}

\def\t{\tau}
\def\s{\sigma}
\def\p{\partial}

\def\mR{{\mathcal{R}}}

\def\D{\Delta}
\def\lk{lk_{\Delta}}
\def\st{st_{\Delta}}
\def\wrt{{\rm with respect to }}
\begin{document}

\title{On a construction of some homology $d$-manifolds}
\author{Biplab Basak$^1$ and Sourav Sarkar}
\date{}
\maketitle
\vspace{-10mm}
\noindent{\small Department of Mathematics, Indian Institute of Technology Delhi, New Delhi 110016, India$^2$.}

\footnotetext[1]{Corresponding author.}

\footnotetext[2]{{\em E-mail addresses:} \url{biplab@iitd.ac.in} (B.
Basak), \url{Sourav.Sarkar@maths.iitd.ac.in} (S. Sarkar).}
\begin{center}
\date{October 10, 2025}
\end{center}

\hrule

\begin{abstract}
The $g$-vector of a simplicial complex contains a lot of information about the combinatorial and topological structure of that complex. Several classification results regarding the structure of normal pseudomanifolds and homology manifolds have been established concerning the value of $g_2$. It is known that when $g_2=0$, all normal pseudomanifolds of dimensions at least three are stacked spheres. In the cases of $g_2=1$ and $2$, all homology manifolds are polytopal spheres and can be obtained through retriangulation or join operations from the previous ones. In this article, we provide a combinatorial characterization of the homology $d$-manifolds, where $d\geq 3$ and $g_2=3$. These are spheres and can be obtained through operations such as joins, some retriangulations, and connected sums from spheres with $g_2\leq 2$. Furthermore, we have presented a structural result on prime normal $d$-pseudomanifolds with $g_2=3$. 
\end{abstract}

\noindent {\small {\em MSC 2020\,:} Primary 05E45; Secondary 55U10, 57Q05, 57Q15, 57Q25.

}

\medskip
\section{Introduction}
The combinatorial characterization of certain special classes of simplicial complexes, namely homology manifolds and normal pseudomanifolds, is being developed in terms of the $g$-theorem. The variation of the third component of the $g$-vector of a $d$-dimensional simplicial complex $\D$, denoted as $g_2(\D):=f_1(\D)-(d+1)f_0(\D)+\binom{d+2}{2}$, has played a significant role in this development. Walkup's work \cite{Walkup} established that for a triangulated 3-manifold $\D$, $g_2(\D)\geq 0$, and equality holds if and only if $\D$ is a triangulation of a stacked sphere. Over time, this concept was gradually extended to encompass all normal pseudomanifolds through various studies \cite{Barnette1, Barnette2, Kalai, Fogelsanger}. Another proof of the non-negativity of $g_2$ can be found in \cite{BagchiDatta, Gromov}. Additionally, based on \cite{Kalai}, we can deduce the inequality $g_2(\D)\geq g_2(\lk \s)$, where $\D$ is a normal $d$-pseudomanifold of dimension $d\geq 3$, and $\sigma$ is a face of co-dimension three or more. As a consequence, $g_2(\D)\geq 0$ for any normal $d$-pseudomanifold $\D$. Furthermore, the equality $g_2(\D)=0$ holds if and only if $\D$ is a stacked sphere.

Nevo and Novinsky \cite{NevoNovinsky} provided a combinatorial description of homology $d$-manifolds with $g_2 = 1$, while Zheng \cite{Zheng} gave a similar characterization for the case $g_2 = 2$. Both characterizations focus on a class of simplicial complexes known in the literature as {\em prime simplicial complexes}. A pure simplicial complex $\Delta$ of dimension $d$ is said to be {\em prime} if it has no missing facets, where a {\em missing facet} refers to a $d$-dimensional simplex $\sigma$ such that $\sigma \notin \Delta$ but all of its faces belong to $\Delta$. Therefore, prime homology manifolds are the homology manifolds that do not have any missing facets.

\begin{Proposition}{\rm \cite{NevoNovinsky}}\label{Nevo}
Let $d\geq 3$, and let $\D$ be a prime homology $d$-sphere with $g_2(\D)=1$.
Then $\D$ is combinatorially isomorphic to either the join of the boundary complex of two simplices, where each simplex has a dimension of at least $2$ and their dimensions add up to $d+1$, or the join of a cycle and the boundary complex of a $(d-1)$-simplex.
\end{Proposition} 
\begin{Proposition}{\rm\cite{Zheng}}\label{Zheng}
If $d\geq 4$, and $\D$ is a prime normal $d$-pseudomanifold with $g_2(\D)=2$, then $\D$ has one of the following structures:
\begin{enumerate}[$(i)$]
\item If $g_2(\lk v)\geq 1$ for each vertex $v\in\D$, then $\D$ can be presented as $\p\s^1\star\p\s^{i}\star\p\s^{d-i}$, where $\s^j$ denote a $j$ simplex, and $2\leq i\leq (d-2)$.
\item If $g_2(\lk v)=0$ for a vertex $v\in\D$, then $\D$ is obtained through one of the following methods:
$(a)$ A central retriangulation of a stacked $d$-sphere along the union of three adjacent facets.
$(b)$ A central retriangulation of a polytopal $d$-sphere with $g_2=1$ along the union of two adjacent facets.
\end{enumerate}
\end{Proposition}

\begin{Proposition}{\rm\cite{Zheng}}\label{Zheng1}
Let $\D$ be a homology $3$-manifold with $g_2(\D)=2$. Then $\D$ has one of the following structures:
\begin{enumerate}[$(i)$]
 \item  $\D$ is an octahedral $3$-sphere, i.e., $\D=\p(a_1b_1)\star\p(a_2b_2)\star\p(a_3b_3)\star\p(a_4b_4)$,
  \item $\D$ is obtained through a central retriangulation of a polytopal $3$-sphere with $g_2 = 1$  along the union of two adjacent facets.
\end{enumerate}
\end{Proposition}

 A  homology manifold with $g_2=1$ can be expressed as the connected sum of one or more homology spheres with lower values of $g_2$, or it can be generated by repeatedly stacking prime complexes with $g_2=1$. Similarly, a homology manifold with $g_2=2$ can be constructed by taking the connected sum of one or more homology spheres with lower $g_2$ values or by being formed through repeated stacking of prime complexes with $g_2=2$. The combinatorial operation known as the connected sum yields general homology spheres (or homology manifolds) due to the following result:

 \begin{Proposition}{\rm\cite{NovikSwartz2020}}\label{NovikSwartz2020}
 If $d \geq 3$ and $\D$ is a homology $d$-manifold that is not prime, then $\D$ can be expressed either as a connected sum of prime homology manifolds or as the outcome of a handle addition on a homology manifold.
 \end{Proposition}

It is known from \cite{Walkup} that if $\D$ is a triangulated $3$-manifold with $g_2\leq 9$, then it is a triangulated sphere. The author also offered combinatorial structures of such triangulations. In \cite{BG}, the authors rephrased the same for their purposes. In particular for homology 3-manifold with $g_2=3$, we have the following:
\begin{Proposition}[\rm \cite{BG, Walkup}]\label{main 3 dim}
Let $\Delta$ be a prime homology $3$-manifold with $g_2(\Delta)=3$. Then $\D$ is a triangulated sphere, and is obtained from a triangulated $3$-sphere with $g_2=2$ by one of the following operations:
\begin{enumerate}[$(i)$]
\item  a bistellar 2-move, 

\item an edge contraction,

\item a combination of an edge expansion and a bistellar 2-move.
\end{enumerate}
 \end{Proposition}
 
 In this article, we establish that every homology $d$-manifold with $g_2=3$ is a triangulated sphere. Additionally, we provide a combinatorial description of these triangulations. We initiated the process by considering a prime normal $d$-pseudomanifold with $g_2=3$ and proved that if all the vertex links are prime, then it can be expressed in one of the following ways: as the join of three standard spheres, generated by iteratively one-vertex suspension on a 6-vertex or 7-vertex triangulation of $\mathbb{RP}^2$, or constructed from a $(d+4)$-vertex cyclic combinatorial sphere $C^{d}_{d+4}$ for certain odd values of $d$, or can be derived from triangulated $d$-spheres having $g_2\leq 2$ through various retriangulation operations, including edge expansion, bistellar 1-move, edge contraction, and generalized bistellar 1-move.

In a related work by Zheng \cite{Zheng}, the question was posed whether $\mathbb{RP}^2\star\mathbb{S}^{d-3}$ is the only non-sphere normal $d$-pseudomanifold with $g_2=3$. Our findings affirmatively address this question when the normal pseudomanifold is prime and has exclusively prime vertex links. Below is the main theorem of interest. 



\begin{Theorem}\label{Main6}
Suppose $d\geq 4$, and let $\D$ be a prime homology $d$-manifold with $g_2(\D)=3$. Then $\D$ is a triangulated sphere and has one of the following structures:
\begin{enumerate}[$(i)$]
\item $\D$ is obtained from a triangulated $d$-sphere with $g_2=2$ by the application of a bistellar 1-move and an edge contraction,
\item $\D=\mathbb{S}^{a}_{a+2}\star\mathbb{S}^{b}_{b+2}\star \mathbb{S}^{c}_{c+2}$, where $a,b,$ and $c\geq 1$, and $d\geq 5$,
\item $\D$ is an iterated one-vertex suspension of the cyclic combinatorial sphere $C^{d}_{d+4}$; for some odd $d\geq 5$,
\item $\D$ is obtained by an edge expansion from a triangulated $d$-sphere with $g_2\leq 2$,
\item $\D$ is obtained by a generalized bistellar $1$-move from a triangulated $d$-sphere with $g_2\leq 2$.
\end{enumerate}
\end{Theorem}

The article is structured into three sections. Section 2 revisits fundamental definitions and provides essential structural tools. Section 3 focuses on a detailed study of prime normal $d$-pseudomanifolds. 
 
\section{Preliminaries}
\subsection{Normal pseudomanifolds}
All the simplicial complexes mentioned throughout the article are finite, and the simplices considered are geometric. If $\sigma$ is an $n$-simplex in $\mathbb{R}^m$ for some $m$, defined as the convex hull of $n+1$ affinely independent points $v_0, \dots, v_n$, then the simplex $\sigma$ is denoted by $v_0\dots v_n$. These points $v_0, \dots, v_n$ are called the vertices of $\sigma$, and the set of all vertices of $\sigma$ is denoted by $V(\sigma)$. If a simplex $\tau$ is the convex hull of a subset of $V(\sigma)$, then $\tau$ is called a face of $\sigma$, denoted by $\tau \leq \sigma$. Furthermore, if $\sigma = v_0\cdots v_n$, and $\t=v_0\cdots v_r$, where $r\leq n$ then by $\sigma - \t$ we mean the $(n-r-1)$-simplex $v_{r+1}\cdots v_n$. 

Let $\D$ be a simplicial complex. Any $i$-simplex present in $\D$ is referred to as an {\em $i$-face} (or simply a {\em face}) of $\D$, and a maximal simplex in $\D$ is called a {\em facet}.  We denote by $|\D|$ the union of all simplices in $\D$, together with the subspace topology induced from $\mathbb{R}^m$ for some $m$. A simplicial complex is considered {\em pure} if all its facets have the same dimension. For $0\leq i\leq 3$, we refer to the $i$-faces as vertices, edges, triangles, and tetrahedra of $\D$, respectively. We use the standard notation $\s^i$ to denote a $i$-simplex when necessary. The set of all vertices of $\D$ is denoted as $V(\D)$. The collection of faces of dimension up to $i$ in $\D$ is called the {\em $i$-skeleton} of $\D$ and is represented as $\text{Skel}_{i}(\D)$. In particular, the $1$-skeleton $\text{Skel}_{1}(\D)$ is also called the {\em graph} of $\D$ and is denoted as $G(\D)$. If $V_1\subseteq V(\D)$, then $\D[V_1]$ represents the subcomplex $\{\sigma\in\D : V(\s)\subseteq V_1\}$. Moreover, a simplex $\s$ is considered a {\em missing face} of $\D$ if $\D$ contains all proper faces of $\s$ but $\s$ itself is not a face of $\D$.

 Two simplices $\sigma = u_0u_1\cdots u_k$ and $\tau = v_0v_1\cdots v_l$ in $\mathbb{R}^n$ for some $n\in \mathbb{N}$ are called {\em independent} if the points $u_0,\dots ,u_k,v_0,\dots,v_l$ are affinely independent. In such a case, $u_0\cdots u_kv_0\cdots v_l$ is a $(k + l + 1)$-simplex and is denoted as $\sigma\star\tau$ or $\sigma\tau$. We refer to $\sigma\star\tau$ (or $\sigma\tau$) as the \textit{join} of two simplices $\sigma$ and $\tau$. Two simplicial complexes, $ \D_1$ and $ \D_2$, are independent if $\sigma\tau$ is an $(i+j+1)$-simplex for every $i$-simplex $\sigma\in \D_1$ and $j$-simplex $\tau\in \D_2$. The join $ \D_1\star \D_2$ of two independent simplicial complexes $ \D_1$ and $ \D_2$, is the simplicial complex  defined as $\{\alpha: \alpha\leq \sigma\tau,\hspace{.25cm}\text{where}\hspace{.25cm} \sigma\in \D_1\hspace{.25cm}\text{and}\hspace{.25cm} \tau\in \D_2\}$. In addition, when referring to a simplex $\s$ and a simplicial complex $\D_1$, the notation $\s\star\D_1$ represents the simplicial complex $\{\t:\t\leq\s\}\star\D_1$.  The {\em link} of a face $\sigma$ in $\D$ is defined as the simplicial complex $\{ \gamma\in \D : \gamma\cap\sigma=\emptyset$ and $ \gamma\sigma\in \D\}$ and is denoted by $\lk \sigma$. The {\em star} of a face $\sigma$ in $ \D$  is defined as the simplicial complex $\{\alpha : \alpha\leq\sigma \beta; \beta\in \lk \sigma\}$ and is denoted by $\st \sigma$. The number of vertices in $\lk u$ is referred to as the degree of that vertex and is denoted as $d(u,\D)$ or simply $d(u)$. As a notation, by $C(a_1,\dots,a_m)$, we mean a cycle of length $m$ with vertices $a_1,\dots,a_m$. If the vertex set is not specified, we shall use the notation $C$ to denote a cycle.

If $\D$ is a $d$-dimensional simplicial complex, then its face number vector, or the {\em $f$-vector}, is defined as the $(d+2)$-tuple $(f_{-1}(\D), f_0(\D), \dots, f_d(\D))$, where $f_{-1}(\D)=1$, and $f_i(\D)$ denotes the number of $i$-faces of $\D$ for $0\leq i\leq d$. Another known enumerative tool, namely $(h_0(\D), h_1(\D), \dots, h_{d+1}(\D))$, referred to as the {\em $h$-vector} of $\D$, where each $h_i(\D)$ is a linear functional of the $f$-vectors as follows:
$$ h_i(\D)=\sum_{j=0}^{i}(-1)^{i-j} \binom{d+1-j}{i-j}f_{j-1}(\D),$$
 and we define  $g_i(\D):= h_i(\D)-h_{i-1}(\D)$. In particular, $g_2(\D) = f_1(\D) - (d+1)f_0(\D) + \binom{d+2}{2}$ and $g_3(\D) = f_2 (\D)- df_1 (\D)+ \binom{d+1}{2}f_0(\D) - \binom{d+2}{3}$.
 Given two positive integers $a$ and $i$, there is a unique way to write
$$a=\binom{a_i}{i}+\binom {a_{i-1}}{i-1}+\dots+\binom{a_{j}}{j},$$
where $a_i>a_{i-1}>\dots>a_j\geq j\geq 1$. Then $a^{<i>}$, the $i$-th Macaulay pseudo-power of $a$, is defined as 
$$a^{<i>}:= \binom{a_i+1}{i+1}+\binom {a_{i-1}+1}{i}+\dots+\binom{a_{j}+1}{j+1}.$$

 In \cite{Swartz2004}, Swartz provided a relation between the $g_i$'s of all pure simplicial complexes and their vertex links. This relation was originally stated by McMullen in \cite{McMullen} for shellable complexes.

 \begin{Lemma}{\rm \cite{McMullen, Swartz2004}}\label{g-relations}
If $\D$ is a $d$-dimensional pure simplicial complex, then for every positive integer $k$,
$$ \sum_{v\in\D} g_k(\lk v)=(k+1)g_{k+1}(\D) +(d+2-k)g_k(\D).$$
\end{Lemma}

Let $\D$ be a pure $d$-dimensional simplicial complex, and let $\mathbb{F}$ be a given field. We say that $\D$ is a $d$-dimensional {\em $\mathbb{F}$-homology sphere} if the $i$-th reduced homology groups $\tilde{H}_{i}(\lk \sigma; \mathbb{F})\cong \tilde{H}_{i}(\mathbb{S}^{d-j-1}; \mathbb{F})$ for every face $\sigma\in\D$ (including the empty face) of dimension $j$. We say that $\D$ is a {\em $\mathbb{F}$-homology manifold} if all of its vertex links are $(d-1)$-dimensional $\mathbb{F}$-homology spheres. When the underlying field is specified, we simply refer to them as homology $d$-sphere and homology $d$-manifold, respectively. A pure $d$-dimensional simplicial complex, $\D$, is said to be a {\em normal $d$-pseudomanifold} if it satisfies the following: $(i)$ Every $(d-1)$-face of $\D$ is contained in exactly two facets. $(ii)$ For any two facets $\sigma$ and $\tau$ in $\D$, there exists a sequence of facets $\sigma_0,\sigma_1,\dots,\sigma_m$ in $\D$ such that $\sigma_0=\sigma$, $\sigma_m=\tau$, and $\sigma_i\cap\sigma_{i+1}$ is a $(d-1)$-simplex in $\D$. $(iii)$ The link of any face of co-dimension two or more is connected. Note that a connected homology $d$-manifold is a normal $d$-pseudomanifold.


\begin{Proposition}{\rm \cite{Swartz2014}}
Let $\D$ be a normal $d$-pseudomanifold, where $d\geq 3$. Then $g_3(\D)\leq g_2(\D)^{<2>}$.
\end{Proposition}

\begin{Corollary}\label{g3 bound}
Let $\D$ be a normal $d$-pseudomanifold, where $d\geq 3$. If $g_2(\D)=3$, then $g_3(\D)\leq 4$.
\end{Corollary}

 \begin{Lemma}\label{normal pseudomanifolds with g2 <= 2}
 Let $\D$ be a normal $d$-pseudomanifold with $g_2(\D)\leq2$. Then $\D$ is a triangulation of the $d$-sphere.
 \end{Lemma}
 \begin{proof}
We prove the result by induction on the dimension $d$. The result is trivial for dimensions up to 2. Let $\D$ be a normal $3$-pseudomanifold with $g_2 (\D) \leq 2$. Then  $g_2 (\lk u)\leq 2$ for every vertex $u\in\D$. Since all the normal 2-pseudomanifolds are manifolds and triangulated 2-manifolds with $g_2\leq 2$ are triangulated 2-spheres, all the vertex links of $\D$ are triangulated $2$-spheres. Therefore, $\D$ is a $3$-manifold. Now, according to \Cref{Zheng1}, all the homology $3$-manifolds with $g_2 \leq 2$ are triangulated spheres; hence, $\D$ is a triangulated $3$-sphere.

Suppose that the result is true for the normal pseudomanifolds of dimensions up to $m$. Let $\D$ be a normal $(m+1)$-pseudomanifold with $g_2(\D) \leq 2$. Then the link of every vertex is a normal $m$-pseudomanifold, and $g_2(\lk v) \leq g_2(\D) \leq 2$ for each vertex $v$ in $\D$. By the induction hypothesis, $\lk v$ is a triangulated sphere for each vertex $v\in\D$. Therefore, $\D$ is a triangulated $(m+1)$-manifold with $g_2 (\D)\leq 2$. Hence, according to \Cref{Zheng}, $\D$ is a triangulated $(m+1)$-sphere.
 \end{proof}

It follows from the definition of normal pseudomanifold that the only $d$-dimensional normal pseudomanifold with $d+2$ vertices is the standard sphere $\mathbb{S}^d_{d+2}$, which is the boundary complex of a $(d+1)$-simplex. Various categorizations for normal pseudomanifolds are available, considering specific numbers of vertices and edges. The following particular outcome was initially presented in \cite{BagchiDatta98} and subsequently noticed in \cite{Zheng}.
\begin{Proposition}{\rm \cite{BagchiDatta98, Zheng}}\label{d+3 vertex npm}
If $\D$ is a normal $d$-pseudomanifold with $d+3$ vertices, then $\D=\p\s^{i}\star\p\s^{d+1-i}$, where $1\leq i\leq d$, and $g_2(\D)\leq 1$. Moreover, if $\D$ is prime, then $g_2(\D)=1$. 
\end{Proposition}
\begin{Lemma}\label{prime vertex links has g2>=1}
 Let $d\geq 4$ and $\D$ be a prime normal $d$-pseudomanifold such that $g_2(\D)\geq 1$. If $u$ is a vertex in $\D$ such that $\lk u$ is prime, then $g_2(\lk u)\geq 1$. 
 \end{Lemma}
 \begin{proof}
 Suppose that $g_2(\lk u)=0$. Then $\lk u$ is a stacked $(d-1)$-sphere. If $\lk u$ is the connected sum of the boundary complex of two or more $d$-simplices, then it contains a missing $(d-1)$-simplex, and hence $\lk u$ is not prime. Thus, $\lk u=\p\s^{d}$ for some $d$-simplex $\s^{d}$. Since $\D$ is prime, we have $\s^{d}\in\D$, and hence $\D$ is the boundary complex of a $(d+1)$-simplex. This contradicts the fact that $g_2(\D)\geq 1$. Therefore, we conclude that $g_2(\lk u)\geq 1$.
 \end{proof}
\subsection{Concept of Rigidity}
Let $G=(V,E)$ be a graph. An injective map $f:V\to\mathbb{R}^{d}$ is called a {\em $d$-embedding} of $G$ into $\mathbb{R}^d$. Let us denote the Euclidean distance in $\mathbb{R}^{d}$ by $\|.\|$. A $d$-embedding $\phi$ of $G$ is called {\em rigid} if there exists $\epsilon>0$ such that for any $d$-embedding $\psi$ of $G$ with $\|\phi(u)-\psi(u)\|<\epsilon$ for every vertex $u\in V$ and  $\|\phi(u)-\phi(v)\|=\|\psi(u)-\psi(v)\|$ for every edge $uv\in E$, one has $\|\phi(u)-\phi(v)\|=\|\psi(u)-\psi(v)\|$ for every pair of vertices $(u,v)$ in $V\times V$. A graph $G$ is called {\em generically $d$-rigid} if the set of all rigid $d$-embeddings is an open and dense subset of the set of all $d$-embeddings of $G$. 
\begin{Proposition}{\rm\cite[Gluing Lemma]{AshimowRoth,NevoNovinsky}}\label{Gluing Lemma}
If $G_1$ and $G_2$ are two generically $d$-rigid graphs such that $G_1\cap G_2$ contains at least $d$-vertices, then $G_1\cup G_2$ is a generically $d$-rigid.
\end{Proposition}
For a $d$-embedding $f$ of $G$, a {\em stress} of $G$ corresponding to $f$ is a function $\omega : E\to\mathbb{R}$ such that for every vertex $v\in V$,
\begin{eqnarray*}
\sum_{uv\in E} \omega(uv)[f(v)-f(u)]=0.
\end{eqnarray*}
The collection of all stresses of a graph $G$ corresponding to a given $d$-embedding $f$ forms a real vector space called the {\em stress space} of $G$ corresponding to $f$. If $G$ is a generically $d$-rigid graph, then the rigid $d$-embeddings are referred to as {\em generic maps}. In the case of $G$ being generically $d$-rigid, the dimension of the stress space is independent of the choice of generic maps. A simplicial complex $\D$ is called generically $d$-rigid if its graph $G(\D)$ is generically $d$-rigid. For a generically $d$-rigid simplicial complex $\D$, we denote the stress space of $G(\D)$ corresponding to any generic map by $\mathcal{S}(\D)$, and each element of $\mathcal{S}(\D)$ is called a generic stress of $G(\D)$.

\begin{Proposition}{\rm\cite{AshimowRoth, Kalai}}\label{g2 as dimension of stress space}
If $\D$ is a $d$-dimensional simplicial complex that is generically $(d+1)$-rigid, then $g_2(\D)$ is equal to the dimension of $\mathcal{S}(\D)$. 
\end{Proposition}
If $\omega(uv)\neq 0$ for an edge $uv\in E$, then we say that $uv$ participates in the stress $\omega$. A vertex $u$ participates in a stress $\omega$ if there exists a vertex $v$ such that the edge $uv$ participates in $\omega$.
\begin{Proposition}{\rm\cite[Cone Lemma]{TayWhiteWhiteley,Whiteley}}\label{Cone Lemma}
Let $\D$ be a simplicial complex and denote by $C(\D,v)$, the cone of $\D$ with the cone vertex $v$. Then,

\begin{enumerate}[$(i)$]
 \item $\D$ is generically $d$-rigid if and only if $C(\D,v)$ is generically $(d+1)$-rigid, and
 \item the stress spaces $\mathcal{S}(\D)$ and $\mathcal{S}(C(\D,v))$ are isomorphic real vector spaces. Moreover, if $g_2(\D)\neq 0$ then $v$ participates in a generic stress of $G(C(\D,v))$.
 \end{enumerate}
\end{Proposition}

\begin{Proposition}{\rm \cite{Fogelsanger,Kalai}}\label{g2 of link is bounded by g2 of complex}
Let $\D$ be a normal $d$-pseudomanifold. Then,
\begin{enumerate}[$(i)$]
 \item $\D$ is generically $(d+1)$-rigid, and
 \item for every face $\s$ of co-dimension $3$ or more in $\D$, $g_2(\lk\s)\leq g_2(\D)$.
 \end{enumerate}
\end{Proposition} 

\begin{Proposition}{\rm \cite{BasakSwartz,Kalai}}\label{rigid subgraph}
Suppose $\D$ is $d$-dimensional simplicial complex that is generically $(d+1)$-rigid and $\Omega$ is a $(d+1)$-rigid subcomplex of $\D$. Then $g_2(\Omega)\leq g_2(\D)$.
   \end{Proposition}
\begin{Corollary}\label{g2 bound on rigid graphs and subgraphs}  Let $d\geq 3$, and let $\D$ be a normal $d$-pseudomanifold. Then $g_2(\D[V(\st u)])\leq g_2(\D)$ for every vertex $u$ in $\D$. 
\end{Corollary}
\begin{proof}
 Observe that the link of every vertex in $\D$ is a normal $(d-1)$-pseudomanifold. Therefore, by \Cref{g2 of link is bounded by g2 of complex}, $\lk v$ is a generically $d$-rigid. Hence, by \Cref{Cone Lemma}, $\st v$ is generically $(d+1)$-rigid. Since adding additional edges to a generically $(d+1)$-rigid graph keeps the resultant graph generically $(d+1)$-rigid, the complex $\D[V(\st v)]$ is a generically $(d+1)$-rigid subcomplex of $\D$. Hence, by Proposition \ref{rigid subgraph} we have $g_2(\D[V(\st v)])\leq g_2(\D)$.
\end{proof}

\begin{Corollary}\label{number of missing edges}
 Suppose $d\geq 3$, and let $\D$ be a normal $d$-pseudomanifold and $v$ be a vertex in $\D$. If $g_2(\D)=p$ and $g_2(\lk u)=q$, then $\lk u$ has at most $p-q$ missing edges that are present in $\D$.
\end{Corollary}
\begin{proof}
  Since $\D$ is a normal $d$-pseudomanifold, from Proposition \ref{g2 of link is bounded by g2 of complex}, it follows that $g_2(\lk u)\leq g_2(\D)$. Note that $g_2(\lk u)=g_2(\st u)$. Therefore, $g_2(\st u)=q$. If $\lk u$ contains $m$ missing edges that are present in $\D$, then $g_2(\D[V(\st u)])=g_2(\st u) + m$. Using Corollary \ref{g2 bound on rigid graphs and subgraphs}, we obtain $q+m\leq p$. Hence, $m\leq p-q$. This completes the proof.
\end{proof}

\begin{Corollary} \label{no vertex outside}
 Let $\D$ be a normal $d$-pseudomanifold such that $g_2(\lk u)\geq 1$ for every vertex $u$ in $\D$. If $\Omega$ is a generically $(d+1)$-rigid subcomplex of $\D$ such that $g_2(\D)=g_2(\Omega)$, then $V(\D)=V(\Omega)$.
\end{Corollary}

\begin{proof}
Let $v$ be a vertex which is not in $\Omega$. Since $g_2(\lk v)\geq 1$, it follows from Proposition \ref{Cone Lemma} that there is a generic stress in $\st v$ whose support includes an edge incident to $v$.  The support of such a stress is not contained in $\Omega$. This contradicts the fact that  $g_2(\D)=g_2(\Omega)$. Therefore, $V(\D)=V(\Omega)$.
\end{proof}

\subsection{Structural tools to be used}

\begin{definition}
{\rm  Let $\D$ be a $d$-dimensional pure simplicial complex, and let $uv$ be an edge in $\D$. Suppose $\lk uv =\p\t$, where $\t$ is a $(d-1)$-simplex and $\t\notin\D$. Consider the simplicial complex $\D':=(\D\setminus uv \star\p(\t))\cup (\t\star\p(uv))$. We say that $\D'$ is obtained from $\D$ by a \textit{bistellar $(d-1)$-move}, and conversely, $\D$ is obtained from $\D'$ by a \textit{bistellar $1$-move}. Since $\t\notin\D$, the geometric realizations $|\D|$ and $|\D'|$ are PL-homeomorphic, and $g_2(\D)=g_2(\D')+1$. }
\end{definition}

Note that if $\D$ is obtained from $\D'$ by a bistellar $1$-move, then the geometric realizations $|\D|$ and $|\D'|$ are PL-homeomorphic. Moreover, $g_2(\D)=g_2(\D')+1$.

\begin{definition}
{\rm Let $\D$ be a $d$-dimensional simplicial complex, and let $B$ be a subcomplex of $\D$ that is a simplicial $d$-ball. The {\em central retriangulation} of $\D$ along $B$, denoted as $crtr_{B}(\D)$, is the new complex obtained by removing all of the interior faces of $B$ and then taking the cone on the boundary of $B$, where the cone point is a new vertex $u$.}
\end{definition}
\begin{definition}\label{edge expansion}
%
{\rm Let $\D$ be a $d$-manifold, and let $w$ be a vertex in $\D$. Then $\lk w$ is a $(d-1)$-sphere. Suppose $S$ is a $(d-2)$-dimensional sphere in $\lk w$ that separates $\lk w$ into two portions, $D_1$ and $D_2$, where $D_1$ and $D_2$ are $(d-1)$-dimensional balls with common boundary complex $S$, and $D_1\cup D_2=\lk w$. Consider the simplicial complex $\D':=(\D\setminus \st w)\cup(uv\star S)\cup (u\star D_1)\cup (v\star D_2)$, where $u$ and $v$ are two new vertices, and $uv$ is a new edge. Then $lk_{\D'} u \cap lk_{\D'} v= lk_{\D'} uv$, and we say that $\D$ is obtained from $\D'$ by contracting the edge $uv$ to the vertex $w$. The manifold $\D'$ is said to be obtained from $\D$ by an {\em edge expansion}.
}
\end{definition}

Let $\D$ be a normal $d$-pseudomanifold, and let $uv$ be an edge in $\D$. We say that the edge $uv$ satisfies {\em link condition} in $\D$ if $\lk u\cap\lk v=\lk uv$. In this situation, one can contract the edge $uv$, and we say that the edge $uv$ is {\em admissible for edge contraction} or simply an {\em admissible edge}. Let $\D'$ be obtained from $\D$ by contracting the edge $uv$ to a vertex. Then $\D'$ is also a normal $d$-pseudomanifold and $\D$ is obtained from $\D'$ by an edge expansion. Moreover, if either $\lk u$ or $\lk v$ is a triangulated sphere, then $|\D'|$ is PL-homeomorphic to $|\D|$ (cf. \cite{BGS1}). The relation between $g_2(\D)$ and $g_2(\D')$ depends on the number of vertices in $\lk uv$.

If there is a vertex, say $x$, in $\lk(w)$ such that $\lk(xw)$ is the $(d-2)$-dimensional sphere $S$, then the complex $\D'$ in Definition \ref{edge expansion} is obtained by performing a central retriangulation of $\D$ along $\st(xw)$ with the cone point $u$ (or by stellar subdivision at the edge $xw$ with the vertex $u$; see \cite{Swartz2009} for more details) and renaming $w$ as $v$. Thus, the central retriangulation of $\D$ along the star of an edge (or the stellar subdivision at an edge) is a special case of an edge expansion.

\begin{Remark}\label{combination of edge expansion and bistellar move}

{\rm
Let $\D$ be a $d$-manifold and $w$ be a vertex in $\D$. Suppose there exists a $(d-1)$-simplex $\alpha$ such that $\partial \alpha \subseteq \lk w$. If $\alpha \notin \D$, we retriangulate the $d$-ball $\st w$ by removing $w$ and inserting $\alpha$, then coning off the two spheres formed by $\lk w$ and $\alpha$. Let $\D'$ be the resulting complex. Then $g_2(\D')=g_2(\D)-1$ and $|\D'|\cong |\D|$.  Note that this combinatorial operation is a combination of an edge expansion and a bistellar $(d-1)$-move. Conversely,
suppose $u$ and $v$ are two vertices in $\D'$ such that $uv\not\in \D'$ and $st_{\D'}u\cap st_{\D'} v=\{\sigma: \sigma\leq\alpha\}$, where $\alpha$ is a $(d-1)$-simplex. We retriangulate the $d$-ball $st_{\D'}u\cup st_{\D'} v$ by removing $u,v,$ and $\alpha$, and coning off the boundary of $st_{\D'}u\cup st_{\D'} v$ by $w$ and we get back $\D$. Therefore, $\D$ is obtained from $\D'$ by a combination of a bistellar $1$-move and an edge contraction.
}
\end{Remark}
\begin{Lemma}\label{all non prime vertex links with 1<=g2<=2}
Let $d\geq 4$, and let $\D$ be a prime homology $d$-manifold  with $g_2(\D)=3$. If $u$ is a vertex in $\D$ such that $\lk u$ is not prime, then $\D$ is obtained from a triangulated $d$-sphere with $g_2=2$ by the application of a bistellar 1-move and an edge contraction.
\end{Lemma}
\begin{proof}
It follows from Proposition \ref{NovikSwartz2020} that $\lk u=\D_1\#\D_2$, where $\D_1$ and $\D_2$ are homology $(d-1)$-spheres ($\D_1$ or $\D_2$ need not be prime). Let $\t$ be the $(d-1)$-simplex in $\D_1\cap\D_2$. Since $\t\notin\lk u$ and $\D$ is prime, we have $\t\notin\D$. Now retriangulate $\st u$ by removing $u$ and inserting $\t$ and then coning off the two homology spheres formed by $\lk u$ and $\t$. Let $\D'$ be the resulting complex. Then $g_2(\D')=g_2(\D)-1=2$. By Remark \ref{combination of edge expansion and bistellar move}, $|\D'|\cong |\D|$ and $\D$ is obtained from $\D'$ by a combination of a bistellar $1$-move and an edge contraction. Since $g_2(\D')=2$, from \Cref{normal pseudomanifolds with g2 <= 2} it follows that $\D'$ is a triangulated sphere. This completes the proof.
\end{proof}


In \cite[Section 5]{BagchiDatta}, a general version of the bistellar move is defined. Here we define a specific generalized bistellar move to serve our purpose.  
\begin{definition}
{\rm Let $\D$ be a $d$-dimensional pure simplicial complex and $uv$ be an edge in $\D$. Suppose $\lk uv =\p B$, where $B= \p (xy)\star \t$, $\t$ is a $(d-2)$-simplex, and $\t\notin\D$. Consider the simplicial complex $\D':=(\D\setminus uv\star\p B)\cup ( \p(uv)\star B)$. We say that $\D'$ is obtained from $\D$ by a \textit{generalized bistellar $(d-1)$-move}, and conversely, $\D$ is obtained from $\D'$ by a \textit{generalized bistellar $1$-move}. Since $\t\notin\D$, $|\D|$ and $|\D'|$ are PL-homeomorphic, and $g_2(\D)=g_2(\D')+1$. }
\end{definition}

The suspension of a surface, or more generally, a simplicial complex, is a well-known operation that involves the introduction of two new vertices. Another operation, first introduced in \cite{BagchiDatta}, is referred to as the one-vertex suspension, where only one additional vertex is added instead of two.
\begin{definition}\label{one-vertex  suspension}
{ \rm Let $\D$ be a normal $(d-1)$-pseudomanifold and $v$ be a vertex in $\D$. Consider the  normal $d$-pseudomanifold $\sum_{v,u}\D:=(v\star\{\tau\in\D : v\nleq\tau\})\cup(u\star\D)$, where $u$ is a new vertex.  The complex $\sum_{v,u}\D$ is called the \textit{one-vertex suspension} of $\D$ \wrt the vertex $v$.}
\end{definition}

The geometric realization  of $\sum_{v,u}\D$ is the suspension of $|\D|$. In the one-vertex suspension, $g_2$ values of $\sum_{u,v}\D$ and the vertex links change based on the degree of the vertex $v$. If $v$ is not adjacent to $m$ vertices of $\D$, then $g_2(\sum_{u,v}\D)=g_2(\D)+m$. Since the links of the vertices $u$ and $v$ in $\sum_{u,v}\D$ are combinatorially isomorphic to $\D$, we have  $g_2(lk_{\sum_{u,v}\D} u)=g_2(lk_{\sum_{u,v}\D} v)=g_2(\D)$. Direct computation reveals that the $g_2$ value of any other vertex link, say $lk_{\sum_{u,v}\D}x$, can be computed using the formula $g_2(lk_{\sum_{u,v}\D} x)=g_2(\lk x)+f_0(\lk x-\st xv)$.

 \begin{Lemma}\label{missing edge}
Let $d\geq 4$, and let $\D$ be a normal $d$-pseudomanifold with $g_2(\D)=3$. Let $u$ be a vertex in $\D$ such that $\lk u=C\star\p\s^{d-2}$, where $C$ is an $n$-cycle, $n\geq 4$. If either $\s^{d-2}\notin\D$, or $C$ contains a vertex that is not incident to a diagonal edge in $\D[V(C)]$, then $\D$ can be obtained from a triangulated $d$-sphere with $g_2 \leq 2$ by an edge expansion.
\end{Lemma}
\begin{proof}
If $\s^{d-2}\notin\D$, then for every vertex $v\leq \s^{d-2}$, the edge $uv$  is an admissible edge. Now, consider the case where $C$ contains a vertex not incident to a diagonal edge in $\D[V(C)]$. Let $v_1$ be such a vertex in $C$. Then, $\lk u\cap\lk v_1=\lk uv_1=\p\s^{d-2}\star\p(u_1u_2)$ for some vertices $u_1$ and $u_2$ in $C$. Thus, the edge $uv_1$ is also admissible. Let $\D'$ be the complex obtained after applying an edge contraction to an admissible edge. Then, $g_2(\D')\leq 2$, and $\D$ can be obtained from $\D'$ through an edge expansion. Now, the result follows from Lemma \ref{normal pseudomanifolds with g2 <= 2}.
\end{proof}

 \begin{Lemma}\label{d>4,g2=1}
Let $d \geq 4$, and let $\D$ be a prime normal $d$-pseudomanifold with $g_2(\D) = 3$ such that $\lk v$ is prime for every vertex $v \in \D$. Suppose $u$ is a vertex in $\D$ such that $\lk u = C \star \p\s^{d-2}$, where $C$ is a cycle of length $n \geq 4$. Then, the following two conditions cannot hold simultaneously:
\begin{enumerate}[$(i)$]
\item $\s^{d-2} \in \D$, and
\item every vertex of $C$ is incident to at least one diagonal edge in $\D[V(C)]$.
\end{enumerate}
\end{Lemma}
\begin{proof}
Observe that $g_2(\lk u) = 1$. Therefore, by Corollary \ref{number of missing edges}, $\lk u$ can have at most two missing edges that are present in $\D$.

Suppose both $(i)$ and $(ii)$ hold simultaneously. Since every vertex of $C$ is incident to at least one diagonal edge in $\D[V(C)]$, it follows that $n = 4$. Moreover, there are exactly two missing edges in $\lk u$ that are present in $\D$. Let $e_1$ and $e_2$ be these two edges. Then, $G(\D[V(\lk u)])=G(\lk u)\cup\{e_1,e_2\}$. Furthermore, since $g_2(\lk u) = g_2(\st u) = 1$, we have $g_2(\D) = g_2(\D[V(\st u)])$. By Corollary \ref{no vertex outside}, it follows that $V(\st u) = V(\D)$.

Since $\s^{d-2}\in\D$ and $\s^{d-2}\notin \lk u$, it follows from $V(\st u) = V(\D)$ that $V(\lk \s^{d-2}) \subseteq V(C)$. However, for any vertex $z \in V(\lk \s^{d-2}) \cap V(C)$, we have $\p(u \star \s^{d-2}) \subseteq \lk z$. Since $u \star \s^{d-2} \notin \D$, the link of $z$ in $\D$ is not prime, which leads to a contradiction. Hence, both $(i)$ and $(ii)$ cannot hold simultaneously.
\end{proof}
 
 \section{Proof of the Main Theorem}
In this section, we present the proof of the main result, Theorem~\ref{Main6}, along with the intermediate results required for its proof. The primary focus is on normal pseudomanifolds of dimension at least 4. The discussion begins with normal $d$-pseudomanifolds $\Delta$ satisfying $g_2(\Delta) = 3$ and $g_2(\lk v) = 3$ for every vertex $v$ in $\Delta$. In Lemma~\ref{prime normal 4pseudomanifold}, we show that normal 4-pseudomanifolds do not exhibit this property. Moreover, for $d \geq 5$, a combinatorial description of such normal $d$-pseudomanifolds is provided in Lemma~\ref{d+4 vertex g2=3}.

We then consider the class $\mathcal{R}_d$ of all prime normal $d$-pseudomanifolds $\Delta$, as defined in Definition~\ref{DefHa}, in which at least one vertex $v$ satisfies $g_2(\lk v) \leq 2$. For $\Delta \in \mathcal{R}_d$, combinatorial descriptions are given when $\Delta$ contains a vertex $u$ with $g_2(\lk u) = 2$, as detailed in Lemmas~\ref{d>4,g2=2}, \ref{d>4,g2=2, one edge with g2=0, three adjacent facet}, \ref{d>4,g2=2, one edge with g2=0, two adjacent facet}, and \ref{octahedral 2 sphere, d=4, g2=2}, using specific characterizations of $\lk u$ described in Propositions~\ref{Zheng} and \ref{Zheng1}. On the other hand, Lemma~\ref{d>4, g2=1, D=join}, together with Lemma~\ref{missing edge}, addresses the case where $\Delta$ contains a vertex $u$ with $g_2(\lk u) = 1$. The proof of Lemma~\ref{d>4, g2=1, D=join} relies on almost every preceding result in this section. The main result of this section is Theorem~\ref{Main5}, which provides a combinatorial characterization of prime normal $d$-pseudomanifolds with prime vertex links and $g_2 = 3$. This result is then used to prove the main theorem of the article.





 \begin{definition}{\rm \cite{Grunbaum}}
{\rm Consider the moment curve $M^{d+1}$ in $\mathbb{R}^{d+1}$ defined parametrically by $x(t) = (t, t^2, \ldots, t^{d+1})$. Let $t_1, t_2, \ldots, t_n$ be real numbers indexed in increasing order, and define $v_i = x(t_i)$ for $1 \leq i \leq n$. For $n \geq d + 2$, let $V$ be the set $\{v_1, v_2, \ldots, v_n\}$, and denote by $C(n, d+1)$ the convex hull of $V$. Then $C(n, d+1)$ is a simplicial convex $(d + 1)$-polytope. The boundary complex of $C(n, d+1)$ is called the cyclic $d$-sphere and is denoted as $C^d_{n}$.}
\end{definition}

Note that, both $C(n, d+1)$ and $C^d_{n}$ are $\lfloor\frac{d+1}{2}\rfloor$-neighborly, and a set $U$ of $(d + 1)$ vertices spans a $d$-face of $C(n, d + 1)$ if and only if any two points in $V \setminus U$ are separated in $M^{d+1}$ by an even number of points from the set $U$. In $C^{2c+1}_ {m+1}$, the link of a vertex is isomorphic to $C^{2c}_m$.

\begin{definition}
{\rm A normal $d$-pseudomanifold $\D$ is called irreducible if $f_0(\D)>d+2$, and $\D$ cannot be written in the form $\D= X\star \mathbb{S}^{c}_{c+2}$, where $X$ is a normal pseudomanifold and $c\geq 0$. We say $\D$ is completely reducible if it is the join of one or more standard spheres.}
\end{definition}
\begin{Proposition}{\rm \cite[Theorems 1 and 2]{BagchiDatta98}}\label{reducible}
Let $\D$ be an irreducible normal $d$-pseudomanifold  with $d+4$ vertices. Then $\D$ is either the cyclic combinatorial sphere $C^{d}_{d+4}$ for some odd $d$, or it can be obtained as an iterated one-vertex suspension of either the $6$-vertex real projective space $\mathbb{RP}_6^{2}$ or the cyclic combinatorial sphere $C^{d'}_{d'+4}$, where $d'$ is odd.
\end{Proposition}
\begin{Lemma}\label{d+4 vertex g2=3}
    Let $d\geq 3$, and let $\D$ be a normal $d$-pseudomanifold with $g_2(\D)=3$. If $f_0(\D)=d+4$, then the graph $G(\D)$ is complete.
\end{Lemma}
\begin{proof}
    If the graph $G(\D)$ is not complete, then $f_1(\D)<{d+4\choose 2}$. Therefore,
    
    $$
        g_2(\D)=f_1(\D)-(d+1)f_0(\D)+{d+2\choose 2}<3.
   $$
Hence, the result follows.
\end{proof}
\begin{Lemma}\label{prime normal 4pseudomanifold}
Let $\D$ be a normal $4$-pseudomanifold with $g_2(\D)= 3$. Then there is a vertex $v$ in $\D$ with $g_2(\lk v)\leq 2$.
\end{Lemma}
\begin{proof}  
Given that $\D$ is a $4$-manifold, a beautiful result of Klee \cite{Klee} ( see also \cite{NovikSwartz2009}), known as the
Dehn–Sommerville relations, asserts that $$h_{5-i}(\D)-h_i(\D)=(-1)^{i}\binom{5}{i}(\chi(\D)-\chi(\mathbb{S}^4)).$$
In particular, for $i=2$, we obtain $$h_3(\D)-h_2(\D)=10(\chi(\D)-2),$$ which implies that $g_3(\D)$ is an integer multiple of 10. On the other hand, Corollary \ref{g3 bound} implies that $g_3(\D)\leq 4$. Combining these two results, we have $g_3(\D)\leq 0$. Now, referring to Lemma \ref{g-relations}, we conclude that there must be a vertex in $\D$, say $v$, with $g_2(\lk v)\leq 1$.

Suppose $\D$ is a normal $4$-pseudomanifold that contains a vertex with a non-sphere link such that $g_2(\lk u) = 3$ for every vertex $u\in\D$. From Corollary \ref{g3 bound} we have $g_3(\D)\leq 4$. Therefore, from Lemma \ref{g-relations}, we get $3f_0(\D)\leq 3\cdot 4+4\cdot 3$, i.e., $f_0(\D)\leq 8$. Note that the only 6-vertex normal 4-pseudomanifold is the boundary complex of a 5-simplex, and by \Cref{d+3 vertex npm}, the 7-vertex normal 4-pseudomanifold has a $g_2$ value of at most 1. Therefore, $f_0(\D)=8$.

Suppose $\D$ is reducible, and let $\D=X\star\mathbb{S}^{c}_{c+2}$, where $X$ is a normal pseudomanifold and $c\geq 0$. By a similar argument as in Lemma \ref{d>4, all g2=3} we get $c\geq 2$, and it follows that the dimension of $X$ is at most one. Thus $g_2(\D)=g_2(X\star\mathbb{S}^{c}_{c+2})\leq 1$, which is a contradiction. Therefore, $\D$ is irreducible. Moreover, since $f_0(\D)=8$ and $\D$ is not a triangulated sphere, Proposition \ref{reducible} implies that $\D$ is obtained from $\mathbb{RP}^2_6$ by two times one-vertex suspension. However, the resulting complex in two times one-vertex suspension of $\mathbb{RP}^2_6$ will have at most four vertices with a link that has a $ g_2$ value of 3. This contradicts the hypothesis that $g_2(\lk u) = 3$ for every vertex $u\in\D$. Thus, we can conclude that $\D$ contains at least one vertex, say $v$, with $g_2(\lk v)\leq 2$.
\end{proof}

\begin{Lemma}\label{d>4, all g2=3}
Let $d\geq 5$, and let $\D$ be a normal $d$-pseudomanifold with $g_2(\D)=3$. Suppose $g_2(\lk u)=3$ for every vertex $u\in\D$. Then  $\D$ is one of the following:
\begin{enumerate}[$(i)$]
\item $\D=\mathbb{S}^{a}_{a+2}\star\mathbb{S}^{b}_{b+2}\star \mathbb{S}^{c}_{c+2}$, where $a,b,$ and $c\geq 2$,
\item $\D$ is the cyclic combinatorial sphere $C^{d}_{d+4}$ for some odd $d$,
 \item $\D$ is an iterated one-vertex suspension of either the $6$-vertex real projective space $\mathbb{RP}_6^{2}$, or the cyclic combinatorial sphere $C^{d'}_{d'+4}$, where $d'$ is odd.
 \end{enumerate}
\end{Lemma}
\begin{proof}
Since $g_2(\D)=3$, from Corollary \ref{g3 bound} we have $g_3(\D)\leq 4$. Therefore, based on the inequality $g_3(\D)\leq 4$ and utilizing Lemma \ref{g-relations}, we can deduce that $3f_0(\D)\leq 3\cdot 4+3\cdot d$, which simplifies to $f_0(\D)\leq d+4$. Note that the only $(d+2)$-vertex normal 4-pseudomanifold is the boundary complex of a $(d+1)$-simplex, and by \Cref{d+3 vertex npm}, a $(d+3)$-vertex normal $d$-pseudomanifold has a $g_2$ value of at most 1. Therefore, $f_0(\D)=d+4$.

Suppose $\D$ is reducible, and let $\D=X\star\mathbb{S}^{c}_{c+2}$, where $X$ is a normal pseudomanifold and $c\geq 0$. If $c=0$, then there is a vertex, say $u$, in  $\mathbb{S}^{c}_{c+2}$ such that $V(\D)\neq V(\st u)$. Therefore, $\lk u$ contains at most $d+2$ vertices. Using Proposition \ref{d+3 vertex npm}, we get $g_2(\lk u)\leq 1$, which is a contradiction. On the other hand, if $c=1$, then for a vertex, say $v$, in $\mathbb{S}^{c}_{c+2}$, $\lk v$ contains a missing edge that is present in $D$, contradicting Corollary \ref{number of missing edges}. 
Therefore, $c\geq 2$. Furthermore, $X$ is a $(d-c-1)$-dimensional normal pseudomanifold with $d-c+2$ vertices. Thus, by Proposition \ref{d+3 vertex npm}, $X$ is the join of two standard spheres, say $\mathbb{S}^{a}_{a+2}$ and $\mathbb{S}^{b}_{b+2}$, where $a$ and $b$ are non-negative integers. Therefore, $\D=\mathbb{S}^{a}_{a+2}\star \mathbb{S}^{b}_{b+2}\star \mathbb{S}^{c}_{c+2}$. If either $a$ or $b$ is in $\{0,1\}$, then by the same argument as above, we get a contradiction. Hence $a,b\geq 2$.

If $\D$ is irreducible, then the conclusion follows from Proposition \ref{reducible}.
\end{proof}

Let $d\geq 4$ and $\D$ be a prime homology $d$-manifold with $g_2(\D)=3$. Recall from Lemma \ref{all non prime vertex links with 1<=g2<=2} that if $\D$ contains a vertex $v$ such that $\lk v$ is not prime, then $\D$ can be obtained from a triangulated sphere with $g_2 = 2$ through a bistellar 1-move and an edge contraction. On the other hand, if $\D$ contains a vertex $u$ for which $\lk u=C\star\p\s^{d-2}$, where $C$ is an $n$-cycle, $n\geq 4$, then Lemma \ref{d>4,g2=1} ensures that either $\s^{d-2} \notin \D$, or $C$ contains a vertex that is not incident to any diagonal edge in $\D[V(C)]$. Therefore, by Lemma \ref{missing edge}, $\D$ is obtained from a triangulated $d$-sphere with $g_2 \leq 2$ by an edge expansion. Furthermore, assuming $g_2(\lk v)=3$ for every vertex $v\in\D$, one has $d\geq 5$ (from Lemma \ref{prime normal 4pseudomanifold}). Therefore, the structure of $\D$ can be deduced from Lemma \ref{d>4, all g2=3}. To address the remaining aspects, i.e., when $\D$ contains a vertex, say $v$, such that $g_2(\lk v)\leq 2$, we will examine the following class.

\begin{definition}\label{DefHa}
For $d\geq 4$, the class $\mR_d$ consists of the prime normal $d$-pseudomanifolds $\D$ such that $g_2(\D)=3$, and $\D$ satisfies the following:

\begin{enumerate}[$(i)$]
\item $\lk w$ is prime for every vertex $w$ in $\D$,

\item if $\D$ contains a vertex $u$ such that $\lk u=C\star\p\s^{d-2}$, where $C$ is an $n$-cycle, then  $n=3$, and

\item there exists a vertex $v$ in $\D$ such that $g_2(\lk v)\leq 2$.
\end{enumerate}
\end{definition}

Note that if $\D$ is a prime normal $d$-pseudomanifold belonging to the class $\mR_d$, then by Lemma \ref{prime vertex links has g2>=1}, we have $g_2(\lk v) \geq 1$ for every vertex $v$ in $\D$.


\begin{Lemma}\label{d>4,g2=2}
Let $d\geq 5$, and let $\D$ be in $\mR_d$. If $u$ is a vertex such that $g_2(\lk u)=2$ and $g_2(\lk uv)\geq 1$ for every vertex $v\in\lk u$, then $\D$ is obtained from a triangulated $d$-sphere with $g_2\leq 2$ by an edge expansion.
\end{Lemma}
\begin{proof}
Since $\lk u$ is a prime normal $d$-pseudomanifold with $g_2(\lk u)=2$, from  Lemma \ref{normal pseudomanifolds with g2 <= 2} it follows that $\lk u$ is a triangulated sphere.  Furthermore, since $g_2(\lk uv)\geq 1$ for every vertex $v\in\lk u$, Proposition \ref{Zheng} $(i)$ implies that $\lk u=\p\s^{1}\star\p\s\star\p\t$, where $\s^1,\s$, and $\t$ are simplices of dimensions $1,i$, and $(d-i-1)$, respectively, for $2\leq i\leq d-3$. Let us take $\p\s^1=\{p,q\}$ (i.e., $\s^1=pq$). We aim to find a vertex $x$ in $\lk u$ such that the edge $ux$ is admissible for edge contraction.

If at least one of $\s^1$, $\s$, and $\t$ is a missing simplex in $\D$, then for any vertex $x$ belonging to that missing simplex, the edge $ux$ is admissible for the edge contraction. Suppose that all three simplices are present in $\D$. Then, $g_2(\D)=g_2(\D[V(\st u)])=3$, and hence, by Corollary \ref{no vertex outside}, it follows that $V(\D)=V(\st u)$.

Since $V(\D)=V(\st u)$ and $u\notin\lk\t$, the number of vertices in $\lk\t$ is at most $i+3$. Therefore, according to Proposition \ref{d+3 vertex npm}, $\lk\t$ falls into one of the following types:
\begin{enumerate}
\item $\p(p\star\s)$ or $\p(q\star\s)$,
\item $\p(pq\star(\s- x_0))$, where $x_0$ is a vertex of $\s$,
\item $\p(pq)\star\p\s$,
\item $\p(x_{1}\cdots x_{j})\star\p(pqx_{j+1}\cdots x_{i+1})$, where $\s=x_1\cdots x_{i+1}$ and $2\leq j\leq {i}$,
\item $\p(px_{1}\cdots x_{j})\star\p(qx_{j+1}\cdots x_{i+1})$, where $\s=x_1\cdots x_{i+1}$ and $1\leq j\leq {i}$.
\end{enumerate}

Suppose that $\lk\t$ is of the form given in type 1. If $\lk\t=\p(p\star\s)$, then $p\star\p\s\subseteq\lk\t$, implying $\p\s\star\t\subseteq\lk p$. Since $\lk u=\p(pq)\star\p\s\star\p\t$, we have $u\star\p\s\star\p\t\subseteq\lk p$. Consequently, $\lk p=(\t\star\p\s)\cup (u\star\p\s\star\p\t)$, and this implies $q\notin\lk p$, leading to $pq\notin\D$, which contradicts the assumption of $\s^1\in\D$. A similar situation arises if $\lk\t=\p(q\star\s)$. Therefore, $\lk \t$ cannot be of type 1.

Let $\lk\t$ be of type 2, 3, or 4. In each case, $\s$ is not in $\lk\t$. We select a vertex $x$ in $V(\s)$ as follows: for type 2, we let $x=x_0$; for type 3, any fixed vertex is selected; and for type 4, we set $x=x_1$. We denote the face $\s-x$ as $\s_x$. Then $\t\star\s_x\star p$ forms a $d$-simplex in $\D$.

Since $\lk\s$ is a normal $(d-i-1)$-pseudomanifold that does not contain the vertex $u$, the number of vertices in $\lk \s$ is at most $d-i+2$. Furthermore, the simplex $\sigma\t$ is not in $\D$, implying that $\lk \s$ falls into one of the following types:
\begin{enumerate}[$(i)$]
\item $\p(pq\star(\t-y_0))$ where $y_0$ is a vertex in $\t$,
\item $\p(pq)\star\p\t$,
\item $\p(y_{1}\cdots y_l)\star\p(pq\star(\t-y_{1}\cdots y_l))$, where $\t=y_{1}\cdots y_{d-i}$.
\end{enumerate}  
Let us choose a vertex $y$ from $V(\t)$ as follows: when $\lk\s=\p(pq\star(\t-y_0))$, we set $y=y_0$; for $\lk\s=\p(pq)\star\p\t$, $y$ can be any vertex; and for $\lk\s=\p(y_{1}\cdots y_l)\star\p(pq\star(\t-y_{1}\dots y_l))$, we take $y=y_1$. For the chosen $y$, we denote the face $\t-y$ as $\t_y$. In any of these three types of $\lk\s$, we will have a $d$-simplex $\t_y\star\s\star p\in\D$. Since $\lk u=\p(pq)\star\p\s\star\p\t$,  $pu\star\s_x\star\t_y$ is a $d$-simplex in $\D$. Therefore, we have three $d$-simplices $\t\star\s_x\star p$, $\t_y\star\s\star p$, and $pu\star\s_x\star\t_y$ in $\D$.  Thus, the link of the $(d-1)$-simplex $p\star\s_x\star\t_y$ contains three distinct vertices: $u$, $x$, and $y$, which leads to a contradiction. Hence, $\lk \t$ cannot be of types 2, 3, or 4.

If $\lk \t$ is of type 5, then $\lk \t$ contains the $i$-simplex $\s$. Hence, $\lk \s$ will be one of the following types:
\begin{enumerate}[$(i)$]
\item $\p(p\star\t)$ or $\p(q\star\t)$,
\item $\p(py_{1}\cdots y_l)\star\p(q\star(\t - y_{1}\cdots y_l))$, where $\t=y_{1}\cdots y_{d-i}$ and $1\leq l\leq (d-i-1)$.
\end{enumerate}
Let $\lk \s=\p(p\star\t)$. Then $p\star\p\t\subseteq\lk\s$, and therefore $\s\star\p\t\subseteq\lk p$. On the other hand, $\lk u=\p(pq)\star\p\s\star\p\t$ implies that $u\star\p\s\star\p\t\subseteq\lk p$. Since $\lk p$ is closed, it follows that $\lk p=(\s\star\p\t)\cup (u\star\p\s\star\p\t)$. Therefore, $q$ is not in $\lk p$, implying that $pq$ is not an edge in $\D$. This contradicts the hypothesis that the simplex $\s^1$ is present in $\D$. A similar situation arises if $\lk\s=\p(q\star\t)$. Therefore, $\lk \s$ is not of type $(i)$.

Let $\lk \s$ be of type $(ii)$. Let us denote the faces $\s - x_1$ and $\t - y_1$ by $\s_{x_1}$ and $\t_{y_1}$, respectively. In this case, $\lk \t$ is of type 5 implies that $p\star\s_{x_1}\star\t$ is a $d$-simplex in $\D$. Further, if $\lk \s$ is of type $(ii)$, then $p\star\s\star\t_{y_1}$ is a $d$-simplex in $\D$. Therefore, $\lk (p\star\s_{x_1}\star\t_{y_1})$ contains the vertices $x_1$ and $y_1$. Since $\lk u=\p(pq)\star\p\s\star\p\t$, we find that $pu\star\s_{x_1}\star\t_{y_1}$ is a $d$-simplex in $\D$. Hence, the link of the $(d-1)$-simplex $p\star\s_{x_1}\star\t_{y_1}$ contains three vertices: $x_1, y_1$, and $u$. 

However, having three vertices in the link of a $(d-1)$-simplex contradicts the properties of a normal $d$-pseudomanifold. Therefore, at least one of $\s^1,\s$, and $\t$ must be a missing face in $\D$. Consequently, there is a vertex $x$ in $\lk u$ such that the edge $ux$ is admissible for edge contraction (note that $\lk u$ is a triangulated sphere).

Let $\D'$ be the resulting complex obtained after edge contraction. Then, $f_0(\D')=f_0(\D)-1$ and $f_1(\D')\leq f_1(\D)-d-2$. Thus, $g_2(\D')\leq 2$, and by Lemma \ref{normal pseudomanifolds with g2 <= 2}, $\D'$ is a triangulated sphere. Moreover, $\D$ is obtained from $\D'$ by an edge expansion.
\end{proof}
\begin{Lemma}\label{octahedral 2 sphere, d=4, g2=2}
Let $\D\in \mR_4$, and let $\D$ contain a vertex $u$ such that $\lk u$ is an octahedral $3$-sphere. Then $\D$ is obtained by a central retriangulation of a triangulated $4$-sphere with $g_2=1$ along the star of an edge. 
\end{Lemma}
\begin{proof}
Since $\lk u$ is an octahedral $3$-sphere, $\lk uv$ is an octahedral $2$-sphere for every vertex $v$ in $\lk u$. Since $g_2(\lk u)=2$, Corollary \ref{number of missing edges} implies that there is at most one missing edge in $\lk u$ that is present in $\D$. Therefore, there are at least two antipodal vertices in $\lk u$ that are not adjacent in $G(\D)$.

\begin{figure}[ht]
\tikzstyle{ver}=[]
\tikzstyle{vertex}=[circle, draw, fill=black!100, inner sep=0pt, minimum width=4pt]
\tikzstyle{edge} = [draw,thick,-]
\centering
\begin{tikzpicture}[scale=0.75]
\begin{scope}[shift={(0,0)}]
\foreach \x/\y/\z in {-.5/0/c,3/0/d,3.5/2/a,0/2/b,1.25/-2/f,1.75/4/g,1.5/1/x,7.5/0/c',11/0/d',11.5/2/a',8/2/b',9.25/-2/f',9.75/4/g',9.5/1/x'}{
\node[vertex] (\z) at (\x,\y){};}
\foreach \x/\y in {a/b,a/d,a/g,a/f,a/x,b/c,b/g,b/f,b/x,d/c,c/x,c/f,c/g,d/x,d/f,d/g,f/x,g/x,a'/b',a'/d',a'/g',a'/f',a'/x',b'/c',b'/g',b'/f',b'/x',d'/c',c'/x',c'/f',c'/g',d'/x',d'/f',d'/g',f'/x',g'/x'}{\path[edge] (\x) -- (\y);}
\foreach \x/\y in {x/x'}{\draw[densely dotted] (\x) -- (\y);}
\foreach \x/\y/\z in {-.9/-.1/c,3.4/-.1/d,3.9/2.2/a,-.4/2.2/b,1.25/-2.5/f,1.75/4.5/g,1.9/1.5/x,7.2/-.1/c,11.5/-.1/d,11.9/2.2/a,7.5/2.2/b,9.25/-2.5/f,9.75/4.5/g,9.9/1.4/y}
\node[ver] () at (\x,\y){$\z$};

\end{scope}
\end{tikzpicture}
\caption{The antipodal vertices $x$ and $y$ in $\lk u$.} \label{fig:2}
\end{figure}
 
 Let us choose two antipodal vertices, say $x$ and $y$, in $\lk u$ such that the edge $xy$ is not present in $G(\D)$, as shown in Figure \ref{fig:2} using a dotted line. Consider the octahedral $2$-sphere $S$ that is the induced subcomplex of $\lk u$ on the vertex set $V(\lk u)\setminus \{x, y\}$. Let $B:=xy\star S$. Then, $B$ is a $4$-ball with $\lk u$ as its boundary.

We replace $st_{\D} u$ in $\D$ with $B$, and denote the resulting complex as $\D'$. Then, $g_2(\D')=1$, and by  Lemma \ref{normal pseudomanifolds with g2 <= 2} $\D'$ is a triangulated $4$-sphere. Moreover, $\D$ can be obtained from $\D'$ by a central retriangulation along $st_{\D'} xy$, which is essentially subdividing the edge $xy$ in $\D'$. More generally, we can describe the process as $\D$ being obtained from $\D'$ through an edge expansion. 
\end{proof}
\begin{Lemma}\label{d>4,g2=2, one edge with g2=0, three adjacent facet}
Let $d\geq 5$, and let $\D$ be in $\mR_d$. Suppose there exists a vertex $u$ such that $g_2(\lk u)=2,$ and $\lk u$ is of type $(a)$ as mentioned in Proposition $\ref{Zheng} \hspace{.1cm}(ii)$. Then $\D$ is obtained from a triangulated $d$-sphere with $g_2\leq 2$ by an edge expansion.
\end{Lemma}
\begin{proof}

According to Proposition \ref{Zheng} $(ii)$, $\lk u$ results from a central retriangulation of the union of three adjacent facets in a stacked $(d-1)$-sphere, say $K$. Let $K:=\p\s^d_1\#\p\s^d_2\#\cdots \#\p\s^d_k$, where each $\s^d_i$ is a $d$-simplex. When $k=1$ (i.e., $K=\p\s^{d}_1$), $\lk u$ contains exactly $d+2$ vertices, leading to $g_2(\lk u)\leq 1$. Therefore, we conclude that $k\geq 2$, indicating that $K$ contains at least one missing $(d-1)$-simplex.

Let $x\leq\s^d_1$ and $y\leq\s^d_k$ be the vertices that are not part of missing $(d-1)$-simplices in $K$. Since $\lk u$ is prime and it is obtained through central retriangulation of $K$ along a ball $B$ that consists of the union of three adjacent facets in $K$, the interior of $B$ contains two $(d-2)$-simplices in such a way that $\lk u$ has no missing $(d-1)$-simplex. Therefore, at least one $(d-2)$-face of each missing $(d-1)$-simplex in $K$ must vanish during the central retriangulation process. This implies that both the vertices $x$ and $y$ are in $B$.

Let $w$ be the new vertex added to form $\lk u$ from $K$ during the central retriangulation such that $\lk uw=\p B$. Since $g_2(\D)=3$ and $g_2(\lk u)=2$, Corollary 
\ref{number of missing edges} implies that $\lk u$ has at most one missing edge that is present in $\D$. If $w$ is not incident to the missing edge of $\lk u$ that is present in $\D$, then $uw$ is an admissible edge. On the other hand, if $w$ is incident to a missing edge in $\lk u$ that is also present in $\D$, then either $ux$ or $uy$ forms an admissible edge. After applying an edge contraction along an admissible edge, let $\D'$ be the resulting complex. Then $g_2(\D')\leq 2$, and hence, by  Lemma \ref{normal pseudomanifolds with g2 <= 2}, $\D'$ is a triangulated $d$-sphere. Moreover, $\D$ is derived from $\D'$ through an edge expansion.
\end{proof}
\begin{Lemma}\label{d>4,g2=2, one edge with g2=0, two adjacent facet}
Let $d\geq 4$, and let $\D$ be in $\mR_d$. If $u$ is a vertex in $\D$ such that $g_2(\lk u)=2$ and $\lk u$ is a central retriangulation of a polytopal $(d-1)$-sphere with $g_2=1$ along the union of two adjacent facets, then $\D$ is one of the following: 

\begin{enumerate}[$(i)$]

 \item $\mathbb{S}^{a}_{a+2}\star \mathbb{S}^{b}_{b+2}\star \mathbb{S}^{c}_{c+2}$, where $a,b$, and $c\geq 1$, and $d\geq 5$,
 
  \item  an iterated one-vertex suspension of $\mathbb{RP}_6^{2}$,    
   
   \item obtained from a triangulated $d$-sphere with $g_2\leq 2$ by an edge expansion or a generalized bistellar $1$-move.
 \end{enumerate}
\end{Lemma}
\begin{proof}
Let $\lk u$ be obtained by centrally retriangulating a polytopal $(d-1)$-sphere, denoted by $K$, with $g_2(K)=1$ along the union of two adjacent facets. Let $w$ represent the newly added vertex during the central retriangulation process. If $K$ is prime, then from  Lemma \ref{normal pseudomanifolds with g2 <= 2} and Proposition \ref{Nevo} it follows that $K$ is either the join of the boundary complex of two simplices, where each simplex has a dimension of at least 2 and their dimensions add up to $d$, or the join of a cycle and the boundary complex of a $(d-2)$-simplex. On the other hand, if $K$ is not prime, then $K$ is obtained from a prime polytopal $(d-1)$-sphere with $g_2 = 1$ through a finite number of facet subdivisions. Furthermore, since $g_2(\lk u)=2$, Corollary \ref{number of missing edges} implies that the number of missing edges in $\lk u$ that are present in $\D$ is at most one.
 
\vst
\noindent \textbf{Case 1:} Assume that $\lk u$ contains a missing edge, say $e$, which is present in $\D$. Then $g_2(\D)=g_2(\D[V(\st u)])=3$, and hence, by \Cref{no vertex outside}, we have $V(\D)=V(\st u)$. Suppose $K$ is prime. Then either $K=\p\s^i\star\p\s^{d-i}$ for some $2\leq i\leq d-2$, or $K=C\star\p\s^{d-2}$, where $C$ is a cycle of length at least 4. 

If $K=\p\s^i\star\p\s^{d-i}$ for some $2\leq i\leq d-2$, then  $V(\D)=V(\st u)$ implies that $f_0(\D)=d+4$. Since $g_2(\D)=3$, by \Cref{d+4 vertex g2=3}, the graph of $\D$ is complete.

Suppose $\D$ is reducible, and let $\D=X\star\mathbb{S}^{c}_{c+2}$, where $c\geq 0$ and $X$ is a normal $(d-c-1)$-pseudomanifold. If $c=0$, then it contradicts the fact that the graph of $\D$ is complete. Therefore, $c\geq 1$.  Since $X$ is a normal $(d-c-1)$-pseudomanifold with $d-c+2$ vertices, by Proposition \ref{d+3 vertex npm}, $X$ is the join of two standard spheres, say $\mathbb{S}^{a}_{a+2}$ and $\mathbb{S}^{b}_{b+2}$ for some non-negative integers $a$ and $b$. Therefore, $\D=\mathbb{S}^{a}_{a+2}\star \mathbb{S}^{b}_{b+2}\star \mathbb{S}^{c}_{c+2}$. If either $a$ or $b$ is $0$, then it contradicts the fact that $G(\D)$ is complete. Thus, both $a$ and $b$ are greater than or equal to 1. Since all of $a,b$, and $c$ are at least 1, we have $d\geq 5$.

If $\D$ is irreducible, then by Proposition \ref{reducible}, $\D$ is either the cyclic combinatorial sphere $C^{d}_{d+4}$ for some odd $d$, or it can be obtained as an iterated one-vertex suspension of either the $6$-vertex real projective space $\mathbb{RP}_6^{2}$, or the cyclic combinatorial sphere $C^{d'}_{d'+4}$, where $d'$ is odd. However, for $d\geq 5$, if $\D$ is a $d$-dimensional manifold that is either the $(d+4)$-vertex cyclic combinatorial sphere $C^{d}_{d+4}$ or an iterated one-vertex suspension of $C^{d'}_{d'+4}$, where $d'$ is odd, then the neighborliness property of the cyclic combinatorial spheres contradicts the existence of a vertex whose link has a $g_2$ values of 2. Further, if $d=4$, and $\D$ is the one vertex suspension of the 7-vertex cyclic combinatorial sphere $C^{3}_{7}$, then by the same argument as in the first paragraph of Lemma \ref{prime normal 4pseudomanifold}, we have $g_3(\D)\leq 0$. Since the $g_2$ of the link of the pair of vertices that participated in the one-vertex suspension is 3 (see the paragraph after Definition \ref{one-vertex  suspension}), therefore, $g_2(\lk v)\geq 1$ for every vertex $v\in \D$ (see Lemma \ref{prime vertex links has g2>=1}) implies that there is no vertex in $\D$ with the link that has a $g_2$-value of 2. Hence, $\D$ is aniterated one-vertex suspension of the $6$-vertex real projective space $\mathbb{RP}_6^{2}$.

Let $K=C\star\p\s^{d-2}$, where $C$ is a cycle of length $n\geq 4$. Here,  $\s^{d-2}\notin\D$; otherwise, $V(\lk\s^{d-2})\subseteq V(C)$, and for any vertex $x\in C\cap\lk\s^{d-2}$, we have $\partial(u\star\s^{d-2})\subseteq \lk x$, contradicting the fact that $\lk x$ is prime. On the other hand, if $n\geq 6$ then there is a vertex in the cycle whose link in $\D$ is a stacked sphere. Therefore, $n\leq 5$. Let $V(C)=\{a_1,\dots,a_m\}$, where $m=4$ or $5$, and $a_ia_{i+1}$ is an edge in $C$. If there is a vertex in $C$, say $x$, that is neither adjacent to $w$ nor incident to $e$, then $d(x)=d+2$, which is not possible. Thus, each vertex of $C$ must be either adjacent to $w$ or incident to $e$ in $\D$. Since $e$ is the only missing edge of $\lk u$ that is present in $\D$, we have $n=4$, and $w$ is incident to $e$. Therefore, there is a vertex, say $c$, in $\s^{d-2}$, which is not adjacent to $w$ in $\D$. Hence, $\lk cu=C\star\p\s^{d-3}$, and $cu$ is an admissible edge.

Suppose that $K$ is not prime, and let $\lk u$ be obtained by subdividing the $(d-2)$-face $\t$, where $lk_{K}\t=\{r,s\}$ for some vertices $r$ and $s$. Note that the $(d-2)$-face $\t$ must participate in all the subdivisions; otherwise, $\lk u$ will not be prime. Moreover, $rs\notin\lk u$. Therefore, if $e=rs$, then $uw$ is an admissible edge (since $\lk u\cap\lk w=\lk uw=\p (rs)\star\p\t$). On the other hand, if $e\neq rs$ then both $ur$ and $us$ are admissible edges. 

Thus, in any situation, there exists an admissible edge for edge contraction. Upon contracting this edge, we obtain a normal pseudomanifold $\D'$ such that $g_2(\D') \leq 2$, and $\D$ is obtained from $\D'$ by an edge expansion. By Lemma \ref{normal pseudomanifolds with g2 <= 2}, $\D'$ is a triangulated sphere. Therefore, the result follows whenever $\lk u$ contains a missing edge that is present in $\D$.

\vst
\noindent \textbf{Case 2:}
Suppose that $\lk u$ has no missing edge that is present in $\D$. If $K$ is not prime, then $\lk u\cap\lk w=\lk uw$ trivially. If $K$ is prime and $K=C\star\p\s^{d-2}$, where $C$ is a cycle of length at least $4$, then $w$ will be incident to at most three vertices of $C$. Let $t$ be a vertex in $C$ that is not adjacent to $w$. Then $\lk t\cap\lk u=\lk tu=\p(xy)\star\p\s^{d-2}$, where $x$ and $y$ are the vertices in $C$ adjacent to $t$.

Now, assume that $K=\p\s^{i}\star\p\s^{d-i}$, where $2\leq i\leq d-2$. Let $\lk u$ be obtained by a central retriangulation of $K$ along the $(d-1)$-ball, denoted as $B$, with a new vertex $w$. Let $\t$ be the $(d-2)$-simplex in the interior of $B$ (the simplex to be subdivided during the central retriangulation of $B$). Note that $B$ contains exactly $d+1$ vertices. Let us take $V(K)\setminus V(B)=\{x\}$, where $x$ is a vertex of $\s^{i}$. If $\s^{i}\notin\D$, then $\lk u\cap\lk u=\lk ux=\p(\s^{i}-x)\star\p\s^{d-i}$, and hence $ux$ is an admissible edge. Now, assume that $\s^i\in\D$.

\vst
\noindent \textbf{Case 2a:} If $\t\notin\D$, then we replace $\st uw=uw\star\p B$ with $B\star\p (uw)$. Let $\D'$ be the resulting complex. Then  $g_2(\D')=2$, and $\D$ is obtained from $\D'$ by a generalized bistellar 1-move. 
\vst
\noindent \textbf{Case 2b:} 
Suppose $\t$ is a face of $\D$. As $\t$ is a $(d-2)$-simplex, its link $\lk \t$ forms a cycle in $\D$. Let $V(B) = V(\t) \cup \{y_1, y_2\}$. By the hypothesis on $x$, both vertices $y_1$ and $y_2$ are in $\s^{d-i}$, and we have $\t = (\s^i - x) \star (\s^{d-i} - y_1 y_2)$. Suppose $y_1$ is in $\lk \t$. Then $\p(u \star \t) \subseteq \lk y_1$. However, $\t$ is a missing $(d-2)$-simplex in $\lk u$, i.e., $u \star \t \notin \D$. This implies that $\lk y_1$ is not prime, which is a contradiction. Therefore, $y_1 \notin \lk \t$. Similarly, $y_2$ and $w$ are not in $\lk \t$.  Since $\lk \t$ is a cycle, it contains at least three vertices. Therefore, $\{y_1,y_2,u,w\}\cap\lk \t=\emptyset$ implies that $\lk \t$ contains at least two more vertices that are not in $\st u$.

Note that $xw \notin \D$. Therefore, $g_2(\lk x) \leq 2$ and $g_2(\lk w) \leq 2$. If $\lk \s^i = \p \s^{d-i}$, then either $y_1 \in \lk \t$ or $y_2 \in \lk \t$, which is a contradiction. Therefore, $\lk \s^i$ contains at least one vertex from $\D \setminus \st u$. Hence, $d(x) \geq d + 3$ and $g_2(\lk x) = 2$.

We claim that there exists a vertex $s$ in $\D$ such that $f_0(\lk s) \geq d + 4$ and $g_2(\lk s) \leq 2$. Once such a vertex exists, the result follows from Lemmas~\ref{d>4,g2=2}–\ref{d>4,g2=2, one edge with g2=0, three adjacent facet}, Case 1, and the first paragraph of Case 2 of this lemma. We conclude the lemma with the proof of the claim.


\vspace{.3cm}
\noindent \textbf{Claim 1:} There exists a vertex $s$ in $\D$ such that $f_0(\lk s)\geq d+4$, and $g_2(\lk s)\leq 2$.

\vspace{.3cm}
\noindent \textit{Proof of Claim 1.}  Note that for any vertex $y \leq \t$, we have $d(y) \geq d + 5$. If $g_2(\lk y) \leq 2$ for some vertex $y \in V(\t)$, then we are done. Now assume that $g_2(\lk y) = 3$ for every vertex $y \in V(\t)$. If $\lk \t$ contains the vertex $x$, then $d(x) \geq d + 4$, and hence we are done. Now assume that $x \notin \lk \t$. Then $V(\lk \t) \subseteq V(\D) \setminus V(\st u)$, and it follows that $f_0(\D) \geq d + 7$. Moreover, by Proposition 2.15 of \cite{BS2}, the graph $G(\D[V(\D \setminus \st u)])$ is complete.

Let $d(p) \leq d + 3$ for every vertex $p \in V(\D) \setminus V(\t)$. Then $d(p) = d + 3$ for every vertex $p \in V(\D) \setminus V(\t)$; otherwise, the existence of a vertex of degree $d + 2$ in $V(\D) \setminus V(\t)$ implies that $f_1(\D)=\frac{1}{2} \sum_{v\in\D} d(v)\leq \frac{1}{2}[(d-1)(f_0(\D)-1)+(d+2)+(f_0-d)(d+3)],$ which simplifies to $g_2(\D) \leq \frac{5}{2}$, contradicting the hypothesis of $g_2(\D)=3$. Therefore, from conditions $(i)$ and $(ii)$ of Definition~\ref{DefHa}, it follows that $g_2(\lk p) \geq 2$ for every vertex $p \in V(\D) \setminus V(\t)$.

Now, according to Lemma~\ref{g-relations}, we have $3(d-1)+2(f_0(\D)-(d-1))\leq 3g_3(\D)+dg_2(\D)$. Since $g_3(\D) \leq 4$ (see Lemma~\ref{g3 bound}), it follows from the last inequality that $2f_0(\D)+d-1\leq 3d+12$. Therefore, $f_0(\D) \leq d + \frac{13}{2}$, which contradicts the fact that $f_0(\D) \geq d + 7$. Thus, there exists a vertex $z \in V(\D) \setminus V(\t)$ such that $d(z) \geq d + 4$. If $g_2(\lk z) \leq 2$, then we are done. Now assume that $g_2(\lk z) = 3$. Then, by the hypothesis, $z \in \{y_1, y_2\}$ with $f_0(\lk z) = f_0(\D) - 1$. Without loss of generality, let $z = y_1$. Thus, we have a set of $d$ vertices, $V(\t) \cup \{y_1\}$, such that $g_2(\lk y) = 3$ for every vertex $y$ in the set.

If $\lk\s^{i}$ contains two or more vertices from $V(\D)\setminus V(\st u)$, then $d(x)\geq d+4$ and we are done. It remains to consider the situation when $\lk\s^{i}$ contains exactly one vertex from $V(\D)\setminus V(\st u)$. However, in such instances, $y_2\notin\lk \t$ implies that $d(y_2)\geq d+4$. 
Therefore, $g_2(\lk y_2) \leq 2$ completes the claim. Assume that $g_2(\lk y_2) = 3$. Then we have a set of $d + 1$ vertices, $V(\t) \cup \{y_1, y_2\}$, such that $g_2(\lk y) = 3$ for every vertex in the set. Lemma~\ref{g-relations} then implies that $3(d+1)+5+[f_0(\D)-d-4]\leq 3d+12$, i.e., $f_0(\D)\leq d+8$. 

If $f_0(\D) = d + 7$, then $\lk \t$ is a $3$-cycle. Moreover, the fact that $x$ is adjacent to a vertex in $\D \setminus \st u$ confirms the existence of a vertex in $\D \setminus \st u$ with degree at least $d + 4$. On the other hand, if $f_0(\D) = d + 8$, then the completeness of the graph $G(\D[V(\D \setminus \st u)])$ implies that each vertex in $\D \setminus \st u$ has degree at least $d + 4$. This completes the proof of Claim 1.
\end{proof}

\begin{Lemma}\label{d>4, g2=1, D=join}
Let $d\geq 4$, and let $\D$ be in $\mR_d$. Suppose that $\D$ contains a vertex $u$  such that $\lk u=\p\s^i\star\p\s^{d-i}$, where $2\leq i\leq (d-2)$.  Then $\D$ is one of the following: 

\begin{enumerate}[$(i)$]

 \item $\mathbb{S}^{a}_{a+2}\star \mathbb{S}^{b}_{b+2}\star \mathbb{S}^{c}_{c+2}$, where $a,b$, and $c\geq 1$, and $d\geq 5$,
 
  \item  an iterated one-vertex suspension of $6$- or $7$-vertex $\mathbb{RP}^{2}$,    
   
   \item obtained from a triangulated $d$-sphere with $g_2\leq 2$ by an edge expansion or a generalized bistellar 1-move.
 \end{enumerate}
\end{Lemma}
\begin{proof}
Since the links of all the vertices are prime, it follows from Lemma \ref{prime vertex links has g2>=1} that $g_2(\lk v)\geq 1$ for every vertex $v$ in $\D$. If at least one of $\s^i$ and $\s^{d-i}$ is a missing face in $\D$, then the edge $ux$ is admissible for every vertex $x$ in that missing face.

Suppose that both $\s^i$ and $\s^{d-i}$ are faces of $\D$. If $\lk \s^i=\p\s^{d-i}$, then $\D=\st u\cup \st \s^i$, so $f_0(\D)=d+3$. Therefore, $\D=\p\s^j\star\p\s^{d-j}$ for some $j$, and hence $g_2(\D)=1$. Thus, $\lk \s^i$ contains at least one more vertex than just those in $V(\s^{d-i})$. It follows that for every vertex $x$ in $\s^i$, $f_0(\lk x)\geq d+3$. Similarly, for any vertex $y$ in $\s^{d-i}$, $f_0(\lk y)\geq d+3$ holds. If there is a vertex $v$ in $V(\s^i)\cup V(\s^{d-i})$ such that $g_2(\lk v)\leq 2$, then the result follows from Lemmas \ref {d>4,g2=2} - \ref{d>4,g2=2, one edge with g2=0, two adjacent facet}. For the remaining part of the proof, we assume that $g_2(\lk w)=3$ for every vertex $w$ in $V(\lk u)$. 

Since $g_2(\D)=3$, by Lemma \ref{g3 bound} we have $g_3(\D)\leq 4$. Furthermore, as discussed in the previous paragraph, $V(\D\setminus\st u)\neq \emptyset$. Let the vertices in $\D\setminus\st u$ be $x_1,\dots,x_s$ with $g_2(\lk x_i)=m_i$. Since $x_i\notin\lk u$, we have $g_2(\lk x_i)\leq 2$. Therefore, Lemma \ref{prime vertex links has g2>=1} provides us with $1\leq m_i\leq 2$. Applying Lemma \ref{g-relations}, we obtain the equation $3(d+2)+g_2(\lk u)+\sum_{i=1}^s m_i= 3g_3(\D)+3d$, which simplifies to $\sum_{i=1}^s m_i+7= 3g_3(\D)$. If $g_3(\D)\leq 2$, then $\sum_{i=1}^s m_i+1\leq 0$, which is not possible. Therefore, $g_3(\D)$ is either $3$ or $4$.

Let $g_3(\D)=4$. Then $\sum_{i=1}^{s} m_i=5$. If $g_2(\lk x_i)=1$ for all $i$, then $s=5$, and it follows that $f_0(\D)=d+8$. Furthermore, the total number of edges in $\D$ is at least $\binom{d+2}{2}+6(d+2)$, and hence $g_2(\D)\geq 6$, which is a contradiction. Therefore, $g_2(\lk x_i)=2$ for some $i$, and the result follows from Lemmas \ref{d>4,g2=2} - \ref{d>4,g2=2, one edge with g2=0, two adjacent facet}.

Let $g_3(\D)=3$. Then $\sum_{i=1}^s m_i=2$. If $s=1$, then $m_1=2$ and $d(x_1)=d+2$, which is a contradiction. Therefore, $s=2$, and hence $m_1=m_2=1$. Since $g_2(\lk w)=3$ for every vertex $w\in\lk u$, by \Cref{no vertex outside}, we have $V(\D)=V(\st w)$ for every vertex $w\in\lk u$. Therefore, the total number of edges connecting the vertices of $\lk u$ and the vertices in $\{x_1,x_2,u\}$ is $3(d+2)$. Now, if $x_1x_2\in\D$, then the total number of edges in $\D$ becomes $f_1(\lk u)+3(d+2)+1$, i.e., ${d+2\choose 2}+3d+7$. Hence, $g_2(\D)={d+2\choose 2}+3d+7-(d+1)(d+5)+{d+2\choose 2}=4$, which is a contradiction. Therefore, $x_1x_2\notin\D$. Since the minimum number of vertices in $\lk x_i$ is $(d+2)$, we have, $V(\lk x_i)=V(\lk u)$ for $i=1,2$. Let $\s^i=u_0u_1\cdots u_i$ and $\s^{d-i}=v_0v_1\cdots v_{d-i}$, where $u_k$ and $v_l$ are vertices of $\lk u$. Since $g_2(\lk x_1)=1$ and $\lk x_1$ contains $d+2$ vertices, we have $\lk x_1=\p(u_0u_1\cdots u_kv_0v_1\cdots v_l)\star\p(u_{k+1}\cdots u_i v_{l+1}\cdots v_{d-i})$ for some non-negative integers $l$ and $k$. Let us denote $\t_1= u_0u_1\cdots u_kv_0v_1\cdots v_l$ and $\t_2=u_{k+1}\cdots u_i v_{l+1}\cdots v_{d-i}$. Here $\t_1$ and $\t_2$ are both different from $\s^i$ and $\s^{d-i}$; otherwise, $\D$ would be a suspension of $\p\s^i\star\p\s^{d-i}$, which contradicts the hypothesis $g_2(\D)=3$. Let $\lk x_2=\p\alpha_1\star\p\alpha_2$. By the same arguments, $\alpha_1$ and $\alpha_2$ are different from $\t_1,\t_2,\s^i$ and $\s^{d-i}$. If any one of the simplices in $\{\t_1,\t_2,\alpha_1,\alpha_2\}$ is a missing face of $\D$, then for every vertex in that missing simplex we would have an admissible edge, and we are done. For the remaining part of the proof, we discuss the structure of $\D$ under the following assumption:

\vspace{.3cm}
\noindent \textbf{Assumption 1:} All the simplices in $\{\s^i,\s^{d-i},\t_1,\t_2,\alpha_1,\alpha_2\}$ are present in $\D$. 

{\vskip .25cm}
\noindent \textbf{Claim 1:} The dimensions of $\t_1,\t_2,\alpha_1,\alpha_2,\s^i,$ and $\s^{d-i}$ are at least $3$, and hence $d\geq 6$.\\

If $\dim (\t_1)=1$, then $\lk x_1$ is the suspension of the boundary of a $(d-1)$-simplex, and therefore, it is a stacked sphere. This contradicts the fact that $g_2(\lk x_1)=1$. If $\dim (\t_1)=2$, then $\t_2$ is a $(d-2)$-simplex in $\D$ that is not present in $\lk x_1$. Since $\lk \t_2$ is a cycle, it must contain at least one vertex, say $x$, from $\t_1$. Thus, $\p(x_1\star \t_2)\subseteq \lk x$. However, $x_1\star \t_2$ is not in $\D$, which contradicts the fact that $\lk x$ is prime. Using similar arguments, we can prove that the dimensions of $\t_2,\alpha_1,\alpha_2,\s^i$ and $\s^{d-i}$ are at least 3.

{\vskip .25cm}
\noindent \textbf{Claim 2:} If  $\lk x_1=\p(u_0u_1\cdots u_kv_0v_1\cdots v_l)\star\p(u_{k+1}\cdots u_i v_{l+1}\cdots v_{d-i})$ for some non-negative integers $l$ and $k$, then $\lk x_2=\p(u_0u_1\cdots u_k v_{l+1}\cdots v_{d-i})\star\p(u_{k+1}\cdots u_iv_0v_1\cdots v_l)$.\\

Note that $\lk x_2=\p\alpha_1\star\p\alpha_2$, where $\alpha_1$ and $\alpha_2$ are two missing faces in $\lk x_2$, and each $(d-1)$-face of $\lk x_2$ is of the form $(\alpha_1-a)\star(\alpha_2- b)$, where $a$ and $b$ are vertices of $\alpha_1$ and $\alpha_2$, respectively.  

\vst
\noindent \textbf{Case 1:} Let $k=i$. Then  $\t_1=u_0u_1\cdots u_iv_{0}v_{1}\cdots v_{l}$ and $\t_2=v_{l+1}\cdots v_{d-i}$, with $\dim (\t_2)\geq 3$, and for any ordered pair $(m,n)$ with $0\leq m\leq i$ and $l+1\leq n\leq (d-i)$, $\s_{mn}=( \s^{i}- u_m)\star(\s^{d-i} - v_n)$ is a $(d-1)$-simplex with $\lk \s_{mn}= \{u,x_1\}$. Therefore, $\s_{mn}$ cannot be in $\lk x_2$, for all $(m,n)$ with $0\leq m\leq i$ and $l+1\leq n\leq d-i$. In other words, both the vertices $u_m$ and $v_n$ are either in $V(\alpha_1)$ or in $V(\alpha_2)$. Now, fix an integer $m_0\in\{0,1,\dots.i\}$. Then, for any $n_0\in\{l+1,\dots, d-i\}$, $u_{m_0}v_{n_0}$ is either an edge of $\alpha_1$ or an edge of $\alpha_2$. Without loss of generality, let $u_{m_0}v_{n_0}\leq\alpha_1$. Then $u_{m_0}v_n\leq\alpha_1$ for all $n\in\{l+1,\dots, d-i\}$. Therefore, $u_mv_n\leq\alpha_1$ for all $(m,n)$ with $0\leq m\leq i$ and $l+1\leq n\leq (d-i)$. Thus, the simplex $\beta:=u_0u_1\cdots u_iv_{l+1}\cdots v_{d-i}$ is a face of $\alpha_1$. Since $\beta$ is neither in $\lk u$ nor in $\lk x_1$, it follows that $\lk \alpha_1$ contains no vertex from the set $\{u,x_1,x_2\}$. Hence, $\lk \alpha_1=\p\alpha_2$, and this implies $\D=\st x_2\cup\st \alpha_1$, contradicting the fact that $u$ and $x_1$ are in $\D$. Therefore, $k<i$. By a similar argument, we can prove that $l<d-i$.


\vst
\noindent \textbf{Case 2:}
Assume that $k<i$ and $l<(d-i)$. For an ordered pair $(m,n)$ with $0\leq m\leq k$ and $l+1\leq n\leq (d-i)$, $\s_{mn}:= (\s^{i}- u_m)\star(\s^{d-i} - v_n)$ is a $(d-1)$-simplex with $\lk \s_{mn}= \{u,x_1\}$.  Therefore, $\s_{mn}$ cannot be in $\lk x_2$ for all $(m,n)$ with $0\leq m\leq k$ and $l+1\leq n\leq (d-i)$.  In other words, both the vertices $u_m$ and $v_n$ are either in $V(\alpha_1)$ or in $V(\alpha_2)$. Now, fix an integer $m_0\in\{0,1,\dots k\}$. Then, for any $n\in\{l+1,\dots, d-i\}$, $u_{m_0}v_n$ is either an edge of $\alpha_1$ or an edge of $\alpha_2$. Without loss of generality, let $u_{m_0}v_n\leq\alpha_1$. Then, $u_{m_0}v_n\leq\alpha_1$ for all $n\in\{l+1,\dots, d-i\}$. Therefore, $u_mv_n\leq\alpha_1$, for all $(m,n)$ with $0\leq m\leq k$ and $l+1\leq n\leq (d-i)$. Thus $u_0u_1\cdots u_kv_{l+1}\cdots v_{d-i}\leq\alpha_1$. 
 
 For any ordered pair $(p,q)$ with $k+1\leq p\leq i$ and $0\leq q\leq l$, $\s_{pq}:= (\s^{i}- u_p)\star(\s^{d-i}-v_q)$ is a $(d-1)$-simplex in $\D$ with $\lk \s_{pq}= \{u,x_1\}$.  Therefore, $\s_{pq}$ cannot be a $(d-1)$-simplex in $\lk x_2$ for all $(p,q)$ with $k+1\leq p\leq i$ and $0\leq q\leq l$.  In other words, both the vertices $u_p$ and $v_q$ are either in $V(\alpha_1)$ or in $V(\alpha_2)$. Since $\dim (\alpha_2)\geq 3$, $V(\alpha_2)$ contains at least $4$ vertices from the set $\{u_{k+1},\dots, u_{i},v_0,v_1,\dots, v_l\}$. If $v_j\leq\alpha_2$ for some $j$, then the edge $u_pv_j\leq\alpha_2$ for all $p$ with $k+1\leq p\leq i$. This implies the edges $u_p v_q\leq\alpha_2$ for all pair $(p,q)$ with $k+1\leq p\leq i$ and $0\leq q\leq l$. Hence, $u_{k+1}\cdots u_iv_0v_1\cdots v_l\leq\alpha_2$. Since $\alpha_1$ and $\alpha_2$ have a disjoint set of vertices, we have $\alpha_1=u_0u_1\cdots u_kv_{l+1}\cdots v_{d-i}$ and $\alpha_2=u_{k+1}\cdots u_iv_0v_1\cdots v_l$.
 
 In the remaining part of the proof, we will take $\lk u=\p(u_0u_1\cdots u_i)\star\p(v_0,v_1,\cdots,v_{d-i})$, $\lk x_1=\p(u_0u_1\cdots u_kv_0v_1\cdots v_l)\star\p(u_{k+1}\cdots u_i v_{l+1}\cdots v_{d-i})$ and using Claim 2, $\lk x_2=\p(u_0u_1\cdots u_k v_{l+1}\cdots v_{d-i})\star\p(u_{k+1}\cdots u_iv_0v_1\cdots v_l)$.
 {\vskip .25cm}
\noindent \textbf{Claim 3:} For every vertex $x$ in  $V(\lk u)$, the convex hull of the vertices in $V(\lk u)\setminus\{x\}$ cannot be a facet of $\D$.\\

Let $x=u_j$ for some $j\in\{0,1,\dots,i\}$, and assume that the $d$-simplex $(\s^{i}- u_j)\star\s^{d-i}$ belongs to $\D$. Now, choose a vertex $v_s$ from $\lk u$ and consider the $(d-1)$-simplex $\s_{js}=(\s^{i}- u_j)\star(\s^{d-i}- v_s)$. If $s\leq l$, then $\lk \s_{js}=\{u,x_2,v_s\}$, and if $s\geq l+1$, then $\lk \s_{js}=\{u,x_1,v_s\}$. In both cases, the link of a $(d-1)$ simplex contains three vertices, which is not possible. The same result holds when $x=v_j$ for some $j$.

{\vskip .25cm}
\noindent \textbf{Claim 4:} $\D$ is the iterated one-vertex suspension of $7$-vertex $\mathbb{RP}^2$, and in each step, the one-vertex suspension happens with respect to a graph cone vertex.\\

From Claim 3, it follows that for every vertex $x$ in $\lk u$, there is no $d$-simplex in $\D$ which is the convex hull of the vertices in $V(\lk u)\setminus\{x\}$. Moreover, since $x_1,x_2\notin\lk u$ and $x_1x_2\notin \D$, every $d$-simplex in $\D$ contains exactly one vertex from the set $\{x_1,x_2,u\}$. Therefore, using Claims 1 and 2, for $\lk u=\p(u_0u_1\cdots u_i)\star\p(v_0,v_1,\cdots,v_{d-i})$, we obtain 
 \begin{eqnarray*} \D= u &\star & \p(u_0u_1\cdots u_i)\star\p(v_0,v_1,\cdots,v_{d-i})\cup \\ x_1 &\star & \p(u_0u_1\cdots u_kv_0v_1\cdots v_l)\star\p(u_{k+1}\cdots u_i v_{l+1}\cdots v_{d-i})\cup\\ x_2 &\star & \p(u_0u_1\cdots u_k v_{l+1}\cdots v_{d-i})\star\p(u_{k+1}\cdots u_iv_0v_1\cdots v_l).
\end{eqnarray*}

Let $\D_0=\lk u_0$. Then $\D_0$ and $lk_{\D_0} u_{1}$ are given by, 
\begin{eqnarray*} \D_0= u &\star & \p(u_1\cdots u_i)\star\p(v_0,v_1,\cdots,v_{d-i})\cup \\ x_1 &\star & \p(u_1\cdots u_kv_0v_1\cdots v_l)\star\p(u_{k+1}\cdots u_i v_{l+1}\cdots v_{d-i})\cup\\ x_2 &\star & \p(u_1\cdots u_k v_{l+1}\cdots v_{d-i})\star\p(u_{k+1}\cdots u_iv_0v_1\cdots v_l), \hspace{.25cm} \text{and}
\end{eqnarray*}
 \begin{eqnarray*} lk_{\D_0} u_1= u &\star & \p(u_2\cdots u_i)\star\p(v_0,v_1,\cdots,v_{d-i})\cup \\ x_1 &\star & \p(u_2\cdots u_kv_0v_1\cdots v_l)\star\p(u_{k+1}\cdots u_i v_{l+1}\cdots v_{d-i})\cup\\ x_2 &\star & \p(u_2\cdots u_k v_{l+1}\cdots v_{d-i})\star\p(u_{k+1}\cdots u_iv_0v_1\cdots v_l).
\end{eqnarray*}
Consider the one-vertex suspension of $\D_0$ at the vertex $u_1$ and rename the newly introduced vertex in $\sum_{u_1,u_0}\D_0$ as $u_0$. Then by definition, $\sum_{u_1,u_0}\D_0$ is given by the union of the following, 
\begin{enumerate}[$(i)$]
\item $u_0u_1\star lk_{\D_0} u_1$,
\item $\p(u_0u_1)\star u u_2\cdots u_i\star\p(v_0,v_1,\cdots,v_{d-i})$,
\item $\p(u_0u_1)\star x_1(u_2\cdots u_kv_0v_1\cdots v_l)\star\p(u_{k+1}\cdots u_i v_{l+1}\cdots v_{d-i})$,
\item $ \p(u_0u_1)\star x_2(u_2\cdots u_k v_{l+1}\cdots v_{d-i})\star\p(u_{k+1}\cdots u_iv_0v_1\cdots v_l)$.
\end{enumerate}
Since the union of facets in $\sum_{u_1, u_0}\D_0$ is the same as that of $\D$, we have $\D=\sum_{u_1,u_0}\D_0$. We now proceed with $\D_0$ in the next step. If $k\geq 2$, then by similar arguments, $\D_0$ is the one-vertex suspension of $\D_1:=\lk u_0u_1$ at the vertex $u_2$. 
We repeat the same process for the vertices in $\{u_1,u_2,\dots,u_{k-1},u_{k+1},\dots,u_{i-1},v_0,\dots,v_{l-1},v_{l+1},\dots,v_{d-i-1}\}$. The iteration process will stop after $d-2$ steps, and the resulting complex will be as follows:
\begin{eqnarray*} \D^{*}= u &\star & \p(v_lv_{d-i})\star \p(u_ku_i)\cup \\ x_1 &\star & \p(v_lu_k)\star\p( u_iv_{d-i})\cup\\ x_2 &\star & \p( u_i v_l)\star \p( u_k v_{d-i}).
\end{eqnarray*}
 Therefore, $\D$ is obtained by iterated ($(d-2)$-times) one-vertex suspension from $\D^{*}$. Observe that, $\D^{*}$ is a 7-vertex triangulation of $\mathbb{RP}^2$. This completes the proof.
\end{proof}

\begin{Remark}
{\em If $d\leq 5$ in Lemma \ref{d>4, g2=1, D=join}, then $\lk u=\p\s^2\star\p\s^i$, where $2\leq i\leq 3$. If $g_2(\lk w)=3$ for every vertex $w$ in $\lk u$, then by the same argument as in Claim 1 of Lemma \ref{d>4, g2=1, D=join}, at least one of $\s^2$ and $\s^i$ must be a missing simplex in $\D$ (Assumption 1 does not hold). Therefore, for a vertex, say $x$, of that missing simplex, the edge $ux$ is an admissible edge. Hence $\D$ is obtained by an edge expansion from a homology manifold with $g_2\leq 2$.} 
\end{Remark}
\begin{Theorem}\label{Main5}
Let $d\geq 4$ and $\D$ be a prime normal $d$-pseudomanifold with $g_2(\D)=3$ such that $\lk v$ is prime for every vertex $v\in\D$. Then $\D$ is one of the following:
\begin{enumerate}[$(i)$]
\item $\D=\mathbb{S}^{a}_{a+2}\star\mathbb{S}^{b}_{b+2}\star \mathbb{S}^{c}_{c+2}$, where $a,b,$ and $c\geq 1$, and $d\geq 5$
\item $\D$ is an iterated one-vertex suspension of the cyclic combinatorial sphere $C^{d}_{d+4}$ for some odd $d\geq 5$,
\item $\D$ is an iterated one-vertex suspension of either $6$- or $7$-vertex real projective space $\mathbb{RP}^{2}$,
\item $\D$ is obtained by an edge expansion from a triangulated $d$-sphere with $g_2\leq 2$,
\item $\D$ is obtained by a generalized bistellar 1-move from a triangulated $d$-sphere with $g_2\leq 2$.
\end{enumerate}
\end{Theorem}

\begin{proof}
Suppose $\D\notin \mR_d$. Since the link of every vertex of $\D$ is prime, either $\D$ does not satisfy $(ii)$ of Definition \ref{DefHa},  or $\D$ does not satisfy $(iii)$ of Definition \ref{DefHa}. Suppose $\D$ does not satisfy $(ii)$ of Definition \ref{DefHa}, i.e., there exists a vertex $u$ in $\D$ with $\lk u=C\star\p\s^{d-2}$, where $C$ is a cycle of length $n\geq 4$. Then by Lemma \ref{d>4,g2=1}, either $\s^{d-2}\notin\D$ or $C$ contains a vertex that is not incident to a diagonal edge in $\D[V(C)]$, and the result follows from Lemma \ref{missing edge}. On the other hand, if $\D$ does not satisfy $(iii)$ of Definition \ref{DefHa}, i.e., if $g_2(\lk v)=3$ for every vertex $v$ in $\D$, then from Lemma \ref{prime normal 4pseudomanifold} it follows that $d\geq 5$. Therefore, the result can be concluded from Lemma \ref{d>4, all g2=3}.

 Now, assume that $\D\in \mR_d$. Then, by $(iii)$ of Definition \ref{DefHa}, there exists a vertex, say $v$, in $\D$ such that $g_2(\lk v)\leq 2$. Since $\lk w$ is prime for every vertex $w$ in $\D$, by Lemma \ref{prime vertex links has g2>=1} we have $g_2(\lk w)\geq 1$ for every vertex $w$ in $\D$.

 Let $u$ be a vertex in $\D$ with $g_2(\lk u)=2$. It follows from Lemma \ref{normal pseudomanifolds with g2 <= 2} that $\lk u$ is a triangulated sphere. Moreover, the possible combinatorial structures of $\lk u$ follow from Propositions \ref{Zheng} and \ref{Zheng1} for $d\geq 5$ and $d=4$, respectively. If $d=4$ and $\lk u$ is an octahedral 3-sphere, i.e., if $\lk u$ satisfies $(i)$ of Proposition \ref{Zheng1}, then the result follows from Lemma \ref{octahedral 2 sphere, d=4, g2=2}.  On the other hand, if $\lk u$ is obtained through a central retriangulation of a polytopal $3$-sphere with $g_2 = 1$  along the union of two adjacent facets, i.e., if $\lk u$ satisfies $(ii)$ of Proposition \ref{Zheng1}, then the conclusion follows from Lemma \ref{d>4,g2=2, one edge with g2=0, two adjacent facet}. Now, assume that $d\geq 5$. If $g_2(\lk uv)\geq 1$ for every vertex $v\in\lk u$, i.e., if $\lk u$ satisfies $(i)$ of Proposition \ref{Zheng}, then according to Lemma \ref{d>4,g2=2}, $\D$ is obtained from a triangulated $d$-sphere with $g_2\leq 2$ by an edge expansion. On the other hand, if there exists a vertex $v\in\lk u$ such that $g_2(\lk uv) =0$, i.e., if $\lk u$ satisfies $(ii)$ of Proposition \ref{Zheng}, then the conclusion follows from Lemma \ref{d>4,g2=2, one edge with g2=0, three adjacent facet} or Lemma \ref{d>4,g2=2, one edge with g2=0, two adjacent facet} depending on whether $\lk u$ is of type $(a)$ or type $(b)$ as described in Proposition \ref{Zheng} $(ii)$, respectively.

Finally, assume that $\D \in \mR_d$ and that $\D$ contains a vertex $u$ such that $g_2(\lk u) = 1$. It follows from Lemma~\ref{normal pseudomanifolds with g2 <= 2} and Proposition~\ref{Nevo} that $\lk u$ is a triangulated sphere. Moreover, since $\lk u$ is prime, it is combinatorially isomorphic to either the join of the boundary complex of two simplices, where each simplex has a dimension of at least $2$ and their dimensions add up to $d$, or the join of a cycle and the boundary complex of a $(d-2)$-simplex. However, condition~ $(ii)$ of Definition~\ref{DefHa} implies that $\lk u = \p\s^i \star \p\s^{d - i}$, where $2 \leq i \leq d - 2$. Hence, the conclusion follows from Lemma~\ref{d>4, g2=1, D=join}. This completes the proof.
\end{proof}


\vspace{.25cm}
\noindent {\em Proof of Theorem 1.6:} If $\D$ contains a vertex $u$ such that $\lk u$ is not prime, then it follows from Lemma \ref{all non prime vertex links with 1<=g2<=2} that $\D$ is obtained from a triangulated $d$-sphere with $g_2=2$ by the application of a bistellar 1-move and an edge contraction. Now, assume that all the vertex links are prime. Then the result follows from Theorem \ref{Main5}, with the only exception being that $\D$ can never be an iterated one-vertex suspension of a 6- or 7-vertex $\mathbb{RP}^2$, since it is a homology manifold. Moreover, in each case, $|\D|$ is PL-homeomorphic to a $d$-sphere.

\medskip

 \begin{Corollary}
Let $\D$ be a homology $d$-manifold with $g_2(\D)=3$, where $d\geq 3$. Then $\D$ is a triangulated sphere, which is either the connected sum of triangulated spheres with $g_2\leq 2$, or it is obtained by finitely many facet subdivisions on the triangulated spheres mentioned in Proposition {\rm \ref{main 3 dim}} and Theorem {\rm \ref{Main6}}.
\end{Corollary}

 \begin{Remark}
  { \em
  
  Let $\D$ be a normal 3-pseudomanifold with $g_2(\D)=3$. In \cite{BasakSwartz}, it is shown that when $\D$ is not a homology 3-manifold, then $\D$ is obtained by the one-vertex suspension of a triangulation of $\mathbb{RP}^2$ at a graph cone point and by subdividing facets. If $\D$ is a homology 3-manifold that is not prime, then $\D=\D_1\#\D_2\#\cdots\#\D_k$, where $\D_i$ are prime triangulated $3$-spheres with $\sum_{i=1}^{k}g_2(\D_i)=3$. If $g_2(\D_i)=0$, then $\D_i$ is a stacked sphere. If $g_2(\D_i)=1$, then by Proposition \ref{Nevo}, $\D_i$ is $\p\s^{2}\star C$, where $C$ is a cycle (cf. \cite{NevoNovinsky}). If $g_2(\D_i)=2$, then $\D_i$ is either an octahedral 3-sphere or the stellar subdivision of a triangulated 3-sphere with $g_2=1$ at a triangle (cf. \cite{ZhengThesis, Zheng}). If $g_2(\D_i)=3$, then by \Cref{main 3 dim}, $\D_i$ is obtained by one of the operations of type bistellar 2-move, edge contraction, or a combination of an edge expansion and a bistellar 2-move. Therefore, we have a complete characterization of a normal 3-pseudomanifold with $g_2=3$.
}
  \end{Remark}

\noindent {\bf Acknowledgement:} 
The authors would like to thank the anonymous referees for many useful
comments and suggestions. The first author is supported by Science and Engineering Research Board (CRG/2021/000859). The second author is supported by Prime Minister's Research Fellows (PMRF/1401215) Scheme.

\end{document}